\documentclass[a4paper,11pt,DIV=11,
abstract=on
]{scrartcl}
\usepackage{algorithm, booktabs}
\usepackage{mathtools}
\mathtoolsset{showonlyrefs}
\usepackage{amsmath}
\usepackage{amssymb}
\usepackage{array}
\usepackage{amsthm}
\usepackage{subcaption}
\usepackage{mathdots}
\usepackage{mathrsfs}
\usepackage{algorithm}
\usepackage[noend]{algpseudocode}
\usepackage{graphicx}
\usepackage{color}
\usepackage{listings}
\usepackage{tikz}
\usepackage{hyperref}

\newcommand{\weakly}{\rightharpoonup}\normalfont
\newcommand{\N}{\ensuremath{\mathbb{N}}}

\newcommand{\R}{\ensuremath{\mathbb{R}}}

\newcommand{\X}{\mathbb{X}}

\newcommand{\supp}{\textnormal{supp}}

\newcommand{\dx}{\,\mathrm{d}}

\def\3{\ss}

\newmuskip\pFqmuskip

\newcommand*\pFq[6][8]{
  \begingroup 
  \pFqmuskip=#1mu\relax
  \begingroup\lccode`\~=`\,
  \lowercase{\endgroup\let~}\pFqcomma
  {}_{#2}F_{#3}{\left(\genfrac..{0pt}{}{#4}{#5};#6\right)}%
  \endgroup
}
\newcommand*\pRegFq[6][8]{
  \begingroup 
  \pFqmuskip=#1mu\relax
  \begingroup\lccode`\~=`\,
  \lowercase{\endgroup\let~}\pFqcomma
  {}_{#2}\tilde{F}_{#3}{\left(\genfrac..{0pt}{}{#4}{#5};#6\right)}%
  \endgroup
}

\newcommand{\pFqcomma}{\mskip\pFqmuskip}

\DeclareMathOperator{\Lip}{Lip}

\DeclareMathOperator*{\diam}{diam}

\DeclareMathOperator{\dist}{dist}

\DeclareMathOperator*{\argmin}{argmin}

\DeclareMathOperator{\OT}{OT}

\newtheorem{theorem}{Theorem}[section]
\newtheorem{lemma}[theorem]{Lemma}
\newtheorem{remark}[theorem]{Remark}

\newtheorem{example}[theorem]{Example}
\newtheorem{corollary}[theorem]{Corollary}
\newtheorem{proposition}[theorem]{Proposition}

\begin{document}
\title{From Optimal Transport to Discrepancy}

\author{
Sebastian Neumayer\footnotemark[1]
	\and
Gabriele Steidl\footnotemark[1]
	}

\maketitle

\footnotetext[1]{Department of Mathematics,
	Technische Universit\"at Kaiserslautern,
	Paul-Ehrlich-Str.~31, D-67663 Kaiserslautern, Germany,
	\{name\}@mathematik.uni-kl.de.}

\begin{abstract}
A common way to quantify the ,,distance'' between measures is via their discrepancy, 
also known as maximum mean discrepancy (MMD).
Discrepancies are related to Sinkhorn divergences $S_\varepsilon$ with appropriate cost functions as $\varepsilon \to \infty$.
In the opposite direction, if $\varepsilon \to 0$,  Sinkhorn divergences approach another important distance between measures, namely the Wasserstein distance or more generally optimal transport ,,distance''.
In this chapter, we investigate the limiting process for arbitrary measures on compact sets and Lipschitz continuous cost functions.
In particular, we are interested in the behavior of the corresponding optimal potentials
$\hat \varphi_\varepsilon$, $\hat \psi_\varepsilon$ and $\hat \varphi_K$ appearing in the dual formulation of the Sinkhorn divergences and discrepancies, respectively.
While part of the results are known, we provide rigorous proofs for some relations which we have not found in this generality in the literature.
Finally, we demonstrate the limiting process  by numerical examples and show the behavior of the distances when used for the approximation of measures by point measures in a process called dithering. 
\end{abstract}

\section{Introduction}

The approximation of probability measures based on their discrepancies 
is a well examined problem in approximation and complexity theory \cite{Kuipers:1974la,Matousek:2010kb,Nowak:2010rr}.
Discrepancies appear in  a wide range of applications, e.g., in the derivation of quadrature rules \cite{Nowak:2010rr}, 
the construction of designs \cite{DGS1977}, image dithering and representation \cite{EGNS2019,Graf:2013fk,SGBW2010,TSGSW2011}, 
see also Fig.~\ref{fig:dither_sphere}, generative adversarial networks \cite{DRG2015} and multivariate statistical testing \cite{FJG2008,GBRSS2007,GBRSS2012}.
In the last two applications, they are also called kernel based maximum mean discrepancies (MMDs).

\begin{figure}[t]
	\begin{center}
		\includegraphics[width=0.46\textwidth]{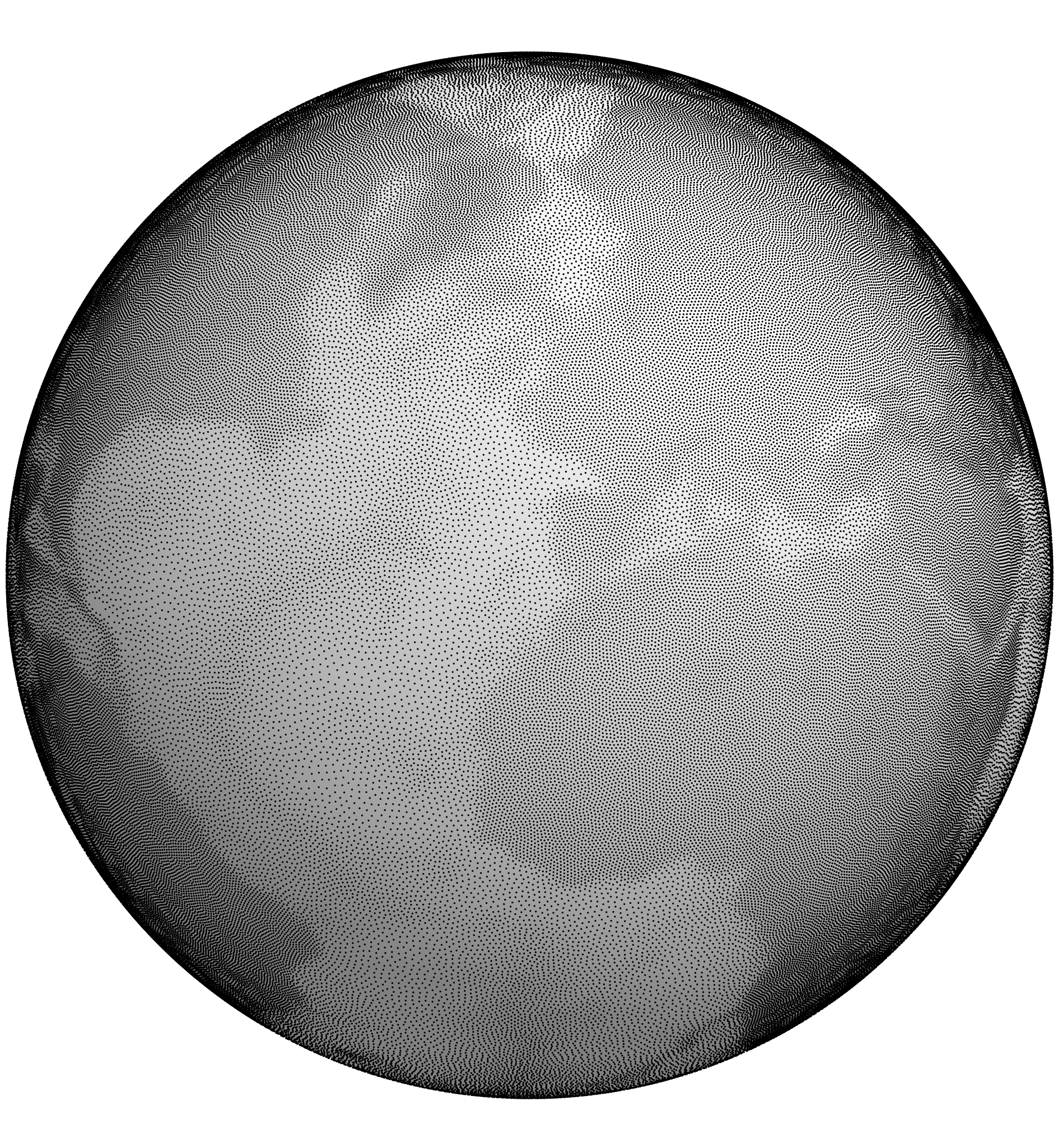}
		\qquad \includegraphics[width=0.46\textwidth]{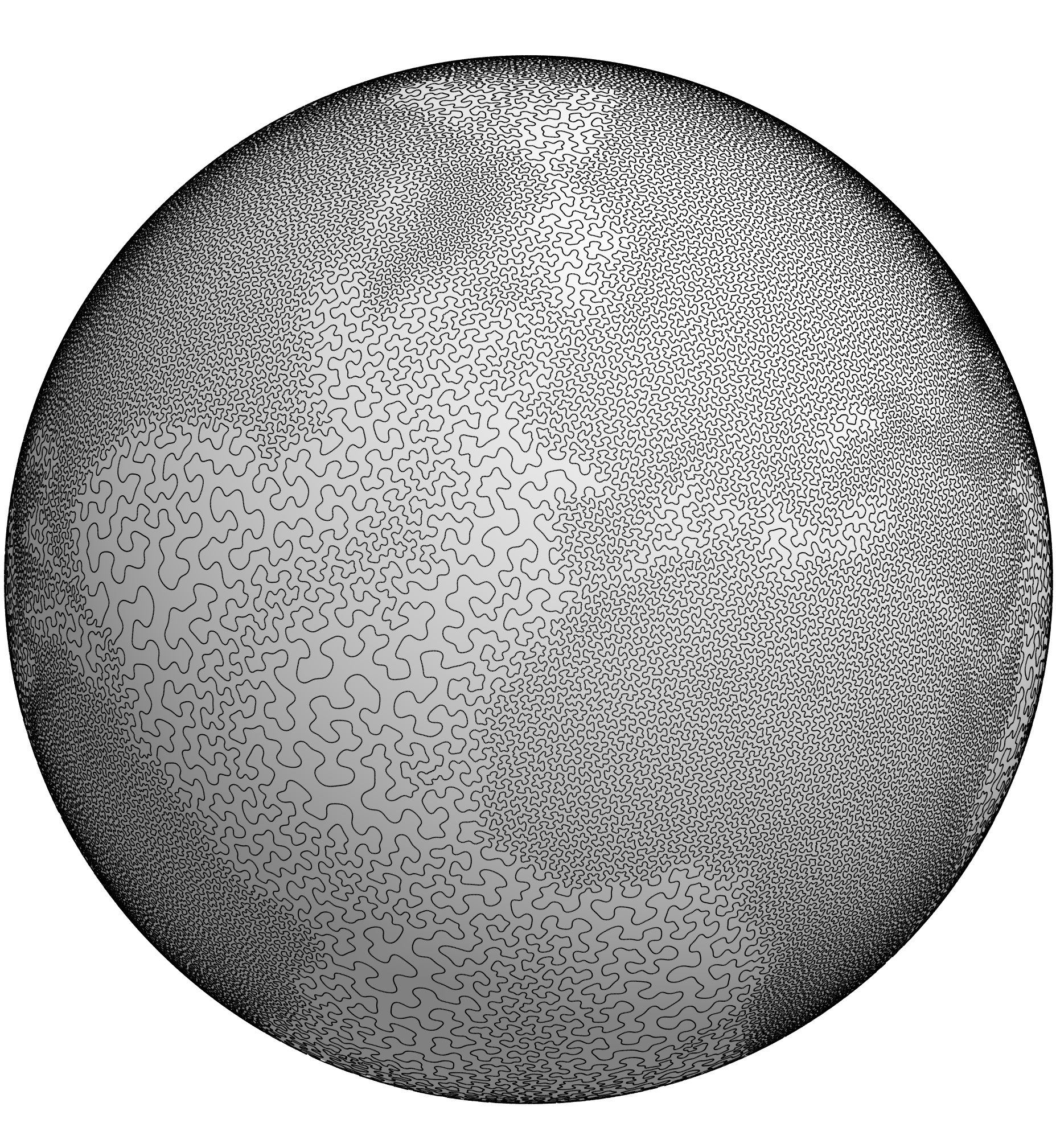}
	\end{center}
	\caption{\label{fig:dither_sphere} 
		Approximation of a measure on $\mathbb S^2$ by an empirical measure \cite{Graf:2013fk} (left) and a measure supported on a curve \cite{EGNS2019} (right)
		using discrepancies as objective function to minimize.}
\end{figure}		

On the other hand, optimal transport (OT) ,,distances'' and in particular Wasserstein distances became very popular for tackling various problems in imaging sciences, graphics or machine learning \cite{CP2019}. There exists a large amount of papers both on the theory and applications of OT,
for image dithering with Wasserstein distances see, e.g., \cite{CCKW2017,GBOD2012,LGKW2018}.

Recently, regularized versions of OT for an efficient numerical treatment, known as Sinkhorn divergences \cite{C2013}, were used as replacement of OT in data science.
Note that such regularization ideas are also investigated in the earlier works \cite{Rus95,Sin64,Wil69,Yul12}.
For appropriately related transport cost functions and discrepancy kernels, the Sinkhorn divergences interpolate between the OT distance if the parameter goes to zero and the discrepancy if it goes to infinity~\cite{FSVATP2018}.
In this chapter, the convergence behavior is examined for general measures on compact sets.
Since cost functions applied in practice are mainly Lipschitz, we restrict our attention to such costs.
This simplifies some proofs, since the theorem of Arzelà--Ascoli can be utilized.
To make the paper self-contained, we provide most of the proofs although some of them are not novel and the corresponding papers are cited in the context.
For estimating approximation rates when approximating measures by those of certain subsets, see, e.g., \cite{Che18,EGNS2019,GCBCP2018,Nowak:2010rr}, the dual form of the discrepancy, respectively of the (regularized) Wasserstein distance, plays an important role. 
Therefore, we are interested in the properties of the optimal dual potentials for varying regularization parameters. 
In Proposition~\ref{lem:PropPot} we prove that the optimal dual potentials 
converge uniformly to certain functions as $\varepsilon \to \infty$.
Then, in Corollary~\ref{cor:conn_disr}, we see that the normalized difference of these limiting functions coincides with the optimal potential in the dual form of the discrepancy if the cost function and the kernel are appropriately related.
This behavior is underlined by a numerical example.

This chapter is organized as follows:
Section \ref{sec:basics_OT} recalls basic results on measures, the Kullback-Leibler (KL) divergence and from convex analysis.
In Section~\ref{sec:discrepancies_OT}, we introduce discrepancies, in particular their dual formulation.
Since these rely on positive definite kernels, we have a closer look at positive definite and conditionally positive definite kernels.
Optimal transport and in particular Wasserstein distances are considered in Section \ref{sec:OT}.
In Section \ref{sec:sinkhorn}, we investigate the limiting processes for the KL regularized OT distances, when the regularization parameter goes to zero or infinity.
Some results in Proposition \ref{prop:conv} are novel in this generality; Proposition \ref{lem:PropPot} seems to be new as well.
Remark \ref{rem:KL-versus-entropy} highlights why the KL divergence should be preferred as regularizer instead of the (neg)-entropy when dealing with non-discrete measures. 
KL regularized OT does not fulfill $\OT_\varepsilon(\mu,\mu) = 0$, which motivates the definition of the Sinkhorn divergence $S_\varepsilon$ in Section \ref{sec:sink_div}.
Further, we prove $\Gamma$-convergence to the discrepancy as $\varepsilon \to \infty$ if the cost function of the Sinkhorn divergence is adapted to the kernel defining the discrepancy.
Section \ref{sec:numerics_OT} underlines the results on the limiting process by numerical examples.
Further, we provide an example on the dithering of the standard Gaussian when Sinkhorn divergences with respect to different regularization parameters $\varepsilon$ are involved. 
Finally, conclusions and directions of future research are given in Section \ref{sec:conclusions_OT}.

\section{Preliminaries} \label{sec:basics_OT}
%
\paragraph{Measures} 
%
Let $\X$ be a compact Polish space (separable, complete metric space) with metric $\dist_\X$. 
By $\mathcal{B}(\X)$ we denote the Borel $\sigma$-algebra on $\X$ 
and by $\mathcal M(\X)$ the linear space of all finite signed Borel measures on $\X$, 
i.e., all $\mu\colon \mathcal{B}(\X) \rightarrow \mathbb R$ satisfying $\mu(\X) < \infty$ and for any sequence 
$\{B_k\}_{k \in \N} \subset \mathcal{B}(\X)$ of pairwise disjoint sets the relation $\mu(\cup_{k=1}^\infty B_k) = \sum_{k=1}^\infty \mu(B_k)$.
In the following, the subset of non-negative measures is denoted by $\mathcal M^+(\X)$.
The \emph{support of a measure} $\mu$ is defined as the closed set 
\[\supp(\mu) \coloneqq \bigl\{ x \in \X: B \subset \X \text{ open, }x \in B  \implies \mu(B) >0\bigr\}.\]

The \emph{total variation} measure of $\mu \in \mathcal M(\X)$ is defined by
\[
|\mu|(B) \coloneqq \sup \Bigl\{ \sum_{k=1}^\infty |\mu(B_k)|:
\bigcup\limits_{k=1}^\infty B_k = B, \, B_k \; \mbox{pairwise disjoint}\Bigr\}.
\]
With the norm $\| \mu\|_{\mathcal M} = |\mu|(\X)$ the space $\mathcal M(\X)$ becomes a Banach space.
By $C(\X)$ we denote the Banach space of continuous real-valued functions on 
$\X$ equipped with the norm $\| \varphi\|_{C(\X)} \coloneqq \max_{x \in \X} |\varphi(x)|$.
The space $\mathcal M(\X)$ can be identified via Riesz' representation theorem with the dual space of $C(\X)$ and the weak-$\ast$ topology on $\mathcal M(\X)$ gives rise to the \emph{weak convergence of measures}.
More precisely, a sequence $\{\mu_k\}_{k \in \N} \subset \mathcal M(\X)$ converges \emph{weakly} to $\mu$ and we write $\mu_k \weakly \mu$, if
\begin{equation}
\lim_{k \to \infty} \int_{\X} \varphi \dx \mu_k = \int_{\X} \varphi \dx \mu \qquad \text{for all } \varphi \in C(\X).
\end{equation}
For a non-negative, finite measure $\mu$ and $p \in [1,\infty)$, let $L^p(\X,\mu)$ 
be the Banach space (of equivalence classes) of complex-valued functions with norm
\[\|f\|_{L^p(\X,\mu)} = \left( \int_\X |f|^p \dx \mu \right)^\frac1p < \infty.\]

A measure $\nu \in \mathcal M(\X)$ is \emph{absolutely continuous} with respect to $\mu$ and we write $\nu \ll \mu$ if for every $A \in \mathcal B(\X)$ with $\mu(A) = 0$ we have $\nu(A) = 0$.
If $\mu, \nu \in \mathcal M^+(\X)$ satisfy $\nu \ll \mu$, then the \emph{Radon-Nikodym derivative} $\sigma_\nu \in L^1(\X,\mu)$ (also denoted by \smash{$\tfrac{\dx \nu}{\dx \mu}$}) exists and $\nu = \sigma_\nu \mu$.
Further, $\mu, \nu \in \mathcal M(\X)$ are \emph{mutually singular} and we write $\mu \perp \nu$ if two disjoint sets $X_\mu, X_\nu \in \mathcal B(\X)$ exist such that $\X = X_\mu \cup X_\nu$ and for every $A \in \mathcal B(\X)$ we have $\mu(A) = \mu(A \cap X_\mu)$ and $\nu(A) = \nu(A \cap X_\nu)$.
For any $\mu,\nu \in \mathcal M^+(\X)$, there exists a unique \emph{Lebesgue decomposition} of $\mu$ with respect to $\nu$ given by $\mu = \sigma_\mu \nu + \mu^\perp$, where $\sigma \in L^1(\X, \nu)$ and $\mu^\perp \perp \nu$.

By $\mathcal P (\X)$ we denote the \emph{set of Borel probability measures} on $\X$,
i.e., non-negative Borel measures with $\mu(\X) = 1$.
This set is \emph{weakly compact}, i.e., compact with respect to the \mbox{weak-$\ast$} topology.
Note that there is an ambiguity in the notation as the above usual weak-$\ast$ convergence is called weak convergence in stochastics.
In Section \ref{sec:OT}, we introduce a metric on $\mathcal P (\X)$ such that it becomes a Polish space.

\paragraph{Convex analysis} 
%
The following can be found, e.g., in \cite{BL2011}.
Let $V$ be a real Banach space with dual $V^*$, i.e., the space of real-valued continuous linear functionals on $V$.
We use the notation $\langle v,x \rangle = v(x)$, $v \in V^*, x \in V$.
For $F\colon V \rightarrow (-\infty,+\infty]$, the \emph{domain} of $F$ is given by
$ \mathrm{dom} F \coloneqq \{x \in V: F(x) \in \mathbb R \}$.
If $\mathrm{dom} F \not = \emptyset$, then $F$ is called \emph{proper}.
The \emph{subdifferential} of $F\colon V \rightarrow (-\infty,+\infty]$ at a point $x_0 \in \mathrm{dom} F$
is defined as
$$
\partial F(x_0) \coloneqq \bigl\{v \in V^*: F(x) \ge F(x_0) + \langle v,x - x_0 \rangle \bigr\},
$$
and $\partial F(x_0) = \emptyset$ if $x_0 \not \in \mathrm{dom} F$. The \emph{Fenchel conjugate}
$F^*\colon V^* \rightarrow (-\infty,+\infty]$ is given by
$$
F^*(v) = \sup_{x \in V} \{ \langle v,x\rangle  - F(x) \}.
$$
If $F\colon V\rightarrow \mathbb (-\infty,+\infty]$ is convex and lower semi-continuous (lsc) at $x \in \mathrm{dom} F$, then
\begin{equation} \label{dual_1}
v \in \partial F(x) \qquad \Leftrightarrow \qquad x \in \partial F^*(v).
\end{equation}
By $\Gamma_0(V)$ we denote the set of proper, convex, lsc functions mapping from $V$ to $(-\infty,+\infty]$.
Let $W$ be another real Banach space.
Then, for $F \in \Gamma_0(V)$,  $G \in \Gamma_0(W)$ and a linear, bounded operator $A\colon V \rightarrow W$
with the property that there exists $x \in \mathrm{dom} F$ such that $G$ is continuous at $A x$,
the following  Fenchel--Rockafellar duality  relation is fulfilled
\begin{equation} \label{primal-dual}
\sup_{x \in V}\bigl\{ - F(-x) - G(Ax)\bigr\}
=
\inf_{w \in W^*} \bigl\{ F^*(A^* w) + G^*(w)\bigr\},
\end{equation}
see \cite[Thm.~4.1, p.~61]{ET1999},
where we consider 
\[\sup_{x \in V}\bigl\{ - F(-x) - G(Ax)\bigr\} = - \inf_{x \in V} \bigl\{  F(-x) + G(Ax)\bigr\}\]
as primal problem with respect to the notation in \cite{ET1999}. 
If the optimal (primal) solution $\hat x$ exists, it is related to any optimal (dual) solution $\hat w$ by
\begin{equation} \label{dual_2}
A \hat x \in \partial G^* (\hat w),
\end{equation}
see \cite[Prop.~4.1]{ET1999}.

\paragraph{Kullback-Leibler divergence} 
%
A function $f \colon [0, +\infty) \to [0, +\infty]$ is called \emph{entropy function}, 
if it is convex, lsc and $\mathrm{dom} f \cap (0,+\infty) \neq \emptyset$.
The corresponding  recession constant is given by
$f'_\infty = \lim_{x \to \infty} \tfrac{f(x)}{x}$.
For every $\mu,\nu \in \mathcal M^+(\X)$ with Lebesgue decomposition $\mu = \sigma_\mu \nu + \mu^\perp$, the \emph{$f$-divergence} is defined as
\begin{equation}\label{eq:FDiv}
D_f(\mu,\nu) = \int_{\X} f \circ \sigma_\mu \dx \nu + f'_\infty\, \mu^\perp(\X).
\end{equation}
In case that $f'_\infty = \infty$ and $\mu^\perp(\X)=0$, we make the usual convention $\infty \cdot 0 = 0$.
The $f$-divergence fulfills $D_f(\mu,\nu) \ge 0$ for all $\mu,\nu \in \mathcal M^+(\X)$ with equality if and only if $\mu = \nu$, and is in general neither symmetric nor satisfies a triangle inequality.
The associated mapping $D_f\colon \mathcal M^+(\X) \times \mathcal M^+(\X) \to [0, +\infty]$ is jointly convex and weakly lsc, 
see \cite[Cor.~2.9]{LMS18}.
The $f$-divergence can be written in the dual form 
\[D_f(\mu,\nu) = \sup_{\varphi \in C(\X)}\int_{\X} \varphi \dx \mu - \int_{\X} f^* \circ \varphi \dx \mu,\]
see \cite[Rem.~2.10]{LMS18}.
Hence, $D_f(\cdot,\nu)$ is the Fenchel conjugate of 
$H \colon C(\X) \to \R$ given by $H(\varphi) \coloneqq \int_{\X} f^* \circ \varphi \dx \nu$.
If $f^*$ is differentiable, we directly deduce from \eqref{dual_1} that
\begin{equation}\label{eq:GradDiv}
\varphi \in \partial_{\mu}D_f(\mu,\nu) 
\quad \Leftrightarrow \quad 
\mu = \nabla H(\varphi) 
\quad \Leftrightarrow \quad 
\mu = \nabla f^* \circ \varphi \, \nu.
\end{equation}

In the following, we focus on the \emph{Shannon-Boltzmann entropy} function and its Fenchel conjugate given by
\[f(x) = x\log(x) - x +1 \quad \mathrm{and} \quad f^*(x) = \exp(x) -1\] 
with the agreement $0 \log 0 = 0$.
The corresponding $f$-divergence is the
\emph{Kullback-Leibler divergence} $\mathrm{KL}\colon{\mathcal M^+}(\X) \times {\mathcal M^+}(\X) \rightarrow \mathbb [0, +\infty]$.
For $\mu,\nu\in {\mathcal M^+}(\X)$ with existing Radon-Nikodym derivative 
$\sigma_\mu = \frac{\dx \mu}{\dx \nu}$ of $\mu$ with respect to $\nu$, formula \eqref{eq:FDiv} 
can be written as
\begin{equation} \label{KLdef}
\mathrm{KL} (\mu,\nu) \coloneqq \int_{\X} \log(\sigma_\mu ) \, \dx \mu + \nu(\X) - \mu(\X).
\end{equation}
In case that the above Radon-Nikodym derivative does not exist, \eqref{eq:FDiv} implies $\mathrm{KL} (\mu,\nu) = + \infty$.
For $\mu,\nu \in \mathcal{P}(\X)$ the last two summands in \eqref{KLdef} cancel each other.
Hence, we have for discrete measures
$\mu = \sum_{j=1}^n \mu_j \delta_{x_j}$ and $\nu = \sum_{j=1}^n \nu_j \delta_{x_j}$ with $\mu_j,\nu_j \ge 0$ and
$\sum_{j=1}^n \mu_j = \sum_{j=1}^n \nu_j = 1$
that
$$
\mathrm{KL}(\mu,\nu) = \sum_{j=1}^n \log \left( \frac{\mu_j}{\nu_j}\right) \mu_j.
$$
Further, the $\mathrm{KL}$ divergence is strictly convex with respect to the first variable.
Due to the Fenchel conjugate pairing
\begin{equation}\label{dual_G}
H(\varphi) = \int_\X \exp(\varphi) -1 \dx \nu  
\quad \mathrm{and} \quad 
H^*(\mu) = \mathrm{KL}(\mu,\nu),
\end{equation}
the derivative relation \eqref{eq:GradDiv} simplifies to
\begin{equation}\label{eq:GradKL}
\varphi \in \partial_{\mu}\mathrm{KL}(\mu,\nu) 
\quad \Leftrightarrow \quad 
\mu = e^\varphi \nu
\quad \Leftrightarrow \quad 
\varphi = \log \Bigl(\frac{\dx \mu}{\dx \nu}\Bigr).
\end{equation}
Finally, note that the KL divergence and  the total variation norm $\Vert \cdot \Vert_\mathcal{M}$ are related by the \emph{Pinsker inequality}
$
\|\mu - \nu\|_{\mathcal M}^2 \le \mathrm{KL} (\mu,\nu). 
$

\section{Discrepancies} \label{sec:discrepancies_OT}
In this section, we introduce the notation of discrepancies and have a closer look at (conditionally) positive definite kernels.
In particular, we emphasize how conditionally positive definite kernels can be modified to  positive definite ones.

Let $\sigma_\X \in \mathcal M(\X)$ be non-negative with $\supp(\sigma_\X) = \X$.
The given definition of discrepancies is based on symmetric, positive definite, continuous kernels.
There is a close relation to general discrepancies related to measures on $\mathcal{B}(\X)$, see \cite{Nowak:2010rr}.
Recall that a symmetric function $K\colon \X \times \X \rightarrow \mathbb R$ is \emph{positive definite} 
if for any finite number $n \in \mathbb N$ 
of points $x_j\in \X$, $j=1,\ldots,n$, the relation
$$
\sum_{i,j=1}^n a_i a_j K(x_i,x_j) \ge 0
$$
is satisfied for all $(a_j)_{j=1}^n \in \R^n$ and \emph{strictly positive definite} if strict inequality holds for all $(a_j)_{j=1}^n \neq 0$.
Assuming that $K \in C(\X \times \X)$ is symmetric, positive definite, 
we know by Mercer's theorem \cite{CS2002,Mer1909,Steinwart:2011it} 
that there exists an orthonormal basis $\{\phi_k: k \in \N\}$ of $L^2(\X,\sigma_\X)$ 
and non-negative coefficients $\{\alpha_k\}_{k \in \N} \in \ell_1$ such that $K$ has the Fourier expansion
\begin{equation} \label{mercer_OT}
K(x,y) = \sum_{k=0}^\infty \alpha_k \phi_k(x)\overline{\phi_k(y)}
\end{equation}
with absolute and uniform convergence of the right-hand side.
If $\alpha_k > 0$ for some $k \in \mathbb N_0$, the corresponding function $\phi_k$ is continuous.
Every function $f\in L^2(\X, \sigma_\X)$ has a Fourier expansion
$$
f = \sum_{k=0}^\infty \hat f_k \phi_k, \quad \hat f_k \coloneqq \int_{\X} f \overline{\phi_k} \dx \sigma_\X.
$$
Moreover, for $k \in \mathbb  N_0$ with $\alpha_k >0$, the \emph{Fourier coefficients} of $\mu \in \mathcal P(\X)$ are well-defined by
\begin{equation}
\hat{\mu}_k\coloneqq \int_{\X}\overline{\phi_k} \dx\mu.
\end{equation}

The kernel $K$ gives rise to a \emph{reproducing kernel Hilbert space} (RKHS).
More precisely, the function space
$$
H_{K} (\X) \coloneqq \Bigl\{f \in L^2(\X, \sigma_\X) : \sum_{k=0}^\infty \alpha_k^{-1} |\hat f_k|^2 < \infty \Bigr\}
$$
equipped with the inner product and the corresponding norm
\begin{equation} \label{norm_rkhs_OT}
\langle f,g \rangle_{H_{K} (\X)} 
= \sum_{k=0}^\infty \alpha_k^{-1} \hat f_k \overline{\hat g_k}, \quad \|f\|_{H_{K} (\X)} = \langle f,f \rangle_{H_{K} (\X)}^\frac12
\end{equation}
forms a Hilbert space with reproducing kernel, i.e.,
\begin{align}
K (x,\cdot) \in H_{K} (\X) \quad \qquad &\mbox{for all} \; x \in \X,\\
f(x) = \left\langle f, K (x,\cdot) \right\rangle_{H_{K} (\X)} \quad &\mbox{for all} \; f \in H_{K} (\X), \; x \in \X. \label{reprod}
\end{align}
Note that $f\in H_K(\X)$ implies $\hat{f}_k=0$ if $\alpha_k=0$, in which case we make the convention $\alpha_k^{-1} \hat f_k=0$ in \eqref{norm_rkhs_OT}.
Indeed, $H_{K} (\X)$ is the closure of the linear span of $\{ K (x_j,\cdot): x_j \in \X \}$ with respect to the norm \eqref{norm_rkhs_OT}.
The space $H_{K} (\X)$ is continuously embedded in $C(\X)$ and hence point evaluations in $H_{K} (\X)$ are continuous.
Since the series in \eqref{mercer_OT} converges uniformly and the functions $\phi_k$ are continuous, the function
$$
\|K(x,\cdot)\|_{H_K(\X)} 
= 
\Bigl \|  \sum_{k=0}^\infty \alpha_k \phi_k(x)\overline{\phi_k(\cdot)} \Bigr\|_{H_K(\X)}
= 
\biggl( \sum_{k=0}^\infty \alpha_k |\phi_k(x)|^2 \biggr)^\frac12
$$
is also continuous so that we have
$\int_\X \|K(x,\cdot)\|_{H_K(\X)} \dx \mu(x) < \infty$.
By the definition of Bochner integrals, see~\cite[Prop.~1.3.1]{HNVW2016}, we have for any $\mu \in \mathcal{P}(\X)$ that
\begin{equation} \label{gg}
\int_\X K(x,\cdot) \dx \mu(x) \in H_K(\X).
\end{equation} 

For $\mu,\nu \in {\mathcal M}(\X)$,
the \emph{discrepancy} $\mathscr{D}_K(\mu,\nu)$ is defined as norm of the linear operator $T\colon H_K \rightarrow \mathbb R$ with $\varphi \mapsto \int_{\X} \varphi  \dx \xi$,
\begin{align}\label{equiv_1_OT}
\mathscr{D}_K(\mu,\nu)
&=
\max_{\| \varphi\|_{ H_{K} (\X) } \le 1}
\int_{\X} \varphi  \dx  \xi,
\end{align}
where $\xi \coloneqq \mu-  \nu$, see \cite{Gnewuch:2012jy,Nowak:2010rr}.
If $\mu_n \weakly \mu$ and $\nu_n \weakly \nu$ as $n\rightarrow \infty$, then also $\mu_n \otimes \nu_n \weakly \mu \otimes \nu$.
Thus, continuity of $K$ implies that $\lim_{n \rightarrow \infty} \mathscr{D}_K(\mu_n,\nu_n) = \mathscr{D}_K(\mu,\nu)$.
Since
\begin{align}
\int_\X \varphi \dx \xi 
= 
\int_\X  \langle \varphi, K(x,\cdot) \rangle_{H_K(\X)} \dx \xi(x)
= 
\Bigl\langle \varphi, \int_\X K(x,\cdot) \dx \xi(x) \Bigr\rangle_{H_K(\X)},
\end{align}
we obtain by Schwarz' inequality that
the optimal dual potential (up to the sign) is given by
\begin{equation}\label{dual_kern}
\hat \varphi_K = \frac{\int_\X K(x,\cdot) \dx \xi(x) }{\Vert \int_\X K(x,\cdot) \dx \xi(x) \Vert_{H_K(\X)}} 
=  \frac{\int_\X K(x,\cdot) \dx \mu(x) - \int_\X K(x,\cdot) \dx \nu(x)}{\Vert K(x,\cdot) \dx \mu(x) 
	- \int_\X  K(x,\cdot) \dx \nu(x) \Vert_{H_K(\X)}} .
\end{equation}
In the following, it is always clear from the context if the Fourier transform of the function 
or the optimal dual potential is meant.	
Further,  Riesz' representation theorem implies
\begin{equation} \label{dual-equiv}
\mathscr{D}_K(\mu,\nu) = \max_{\| \varphi\|_{ H_{K} (\X) } \le 1} \int_{\X} \varphi  \dx \xi 
= \Bigl\| \int_\X K(x,\cdot) 
\dx \xi(x) \Bigr\|_{H_{K}(\X)},
\end{equation}
so that we conclude by Fubini's theorem and \eqref{reprod} that
\begin{align} \label{mercer_1_OT}
\mathscr{D}_K^2(\mu,\nu)
&= \Bigl\| \int_\X K(x,\cdot) 
\dx \xi(x) \Bigr\|_{H_{K}(\X)}^2 
= \int_{\X^2} K \dx(\xi \otimes \xi) \\
&= \int_{\X^2} K \dx(\mu \otimes \mu)
+\int_{\X^2} K\dx(\nu \otimes \nu)  -2\int_{\X^2} K \dx(\mu \otimes \nu).  
\end{align}
By \eqref{mercer_OT}, we finally get
\begin{align} 
\mathscr{D}_K^2(\mu,\nu)
&=\sum_{k=0}^\infty \alpha_k  \big| \hat{\mu}_{k}-\hat{\nu}_{k}  \big|^2, \label{mercer_2_OT}
\end{align}
where the summation runs over all $k\in \mathbb N_0$ with $\alpha_k >0$.

\begin{remark}\upshape\textbf{(Relation to attraction-repulsion functionals)} \label{rem:attraction-repulsion}
	We briefly consider the relation to attraction-repulsion functionals motivated from electrostatic halftoning, see
	\cite{SGBW2010,TSGSW2011}.
	Let $\nu = w \dx x$ be fixed, for example a continuous (normalized) image with gray values in $[0,1]$ represented by $w\colon \X \to [0,1]$, where pure black is the largest value of $w$ and white the smallest one.
	Then, looking for a discrete measure $\mu = \frac{1}{M} \sum_{j=1}^M \delta(\cdot-p_j)$ that approximates $\nu$ by minimizing the squared discrepancy is equivalent to solving the minimization problem
	$$
	\argmin_{p \in \R^M} \biggl\{ \underbrace{\frac{1}{2M} \sum_{i,j=1}^M K(p_i,p_j)}_{\mathrm{repulsion}} 
	- \underbrace{\sum_{i=1}^M \int_{\X } w(x) K(x,p_i)}_{\mathrm{attraction}} \biggr\}.
	$$
	For $K(x,y) = h(\|x-y\|)$ and an decreasing function $h\colon [0,+\infty) \rightarrow \mathbb R$,
	it becomes clear that 
	\begin{itemize}
		\item
		the first term is minimal if the points are far away from each other, implying a \emph{repulsion};
		\item
		the second (negative) term becomes maximal if for large $w(x)$, there are many points positioned in this area; so it can be considered as an \emph{attraction} steered by $w$. 
	\end{itemize}
\end{remark}

\paragraph{Kernels.}  
In this paragraph, we want to have a closer look at appropriate kernels. 
Recall that for symmetric, positive definite kernels $K_i \in C(\X \times \X)$, $i=1,2$, and $\alpha > 0$, 
the kernels $\alpha K_1$, $K_1 + K_2$, $K_1 \cdot K_2$ and $\mathrm{exp}(K_1)$ 
are again positive definite, see \cite[Lems.~4.5 and 4.6]{SC08}.

Of special interest are so-called radial kernels of the form
$$
K(x,y) \coloneqq h\bigl(\dist_{\X} (x,y) \bigr),
$$
where $h\colon [0,+\infty) \rightarrow \mathbb R$. 
In the following, the discussion is restricted to compact sets $\X$ in $\mathbb R^d$ 
and the Euclidean distance $\dist_{\X} (x,y) = \|x-y\|$. 
Many results on positive definite functions on $\mathbb R^d$ go back to Schoenberg \cite{Sch38} and Micchelli \cite{Mi86}.
For a good overview, we refer to \cite{Wendland:2004wd}, where some of the following statements can be found.
Clearly, restricting positive definite kernels on $\mathbb R^d$ to compact subsets $\X$ results in positive definite kernels on $\X$.
The radial kernels related to the Gaussian, which are quite popular in MMDs, and the inverse multiquadric given by 
\begin{align}
h(r) = e^{-r^2/c^2}  \qquad \mathrm{and} \qquad 
h(r) = (c^2 + r^2)^{-p},  \quad c, p >0,
\end{align}
are known to be strictly positive definite on $\mathbb R^d$ for every $d \in \mathbb N$.
Further, the following compactly supported functions $h$ give rise to positive definite kernels in $\mathbb R^d$:
\begin{equation} \label{dist_posdef}
h(r) = (1-r)_+^p, \qquad p \ge \left\lfloor \frac{d}{2} \right\rfloor + 1,
\end{equation}
where $\lfloor a \rfloor$ denotes the largest integer less or equal than $a \in \mathbb R$
and $a_+ \coloneqq \max(a,0)$.

In connection with Wasserstein distances, we are interested in (negative) powers of distances 
$K(x,y) =  \|x-y\|^p$, $p > 0$,  
related to the functions $h(r) = r^p$.
Unfortunately, all these functions are not positive definite!
By \eqref{dist_posdef}, we know that $\tilde K(x,y) = 1-|x-y|$ is positive definite in one dimension $d=1$.
A more  general result for the Euclidean distance is given in the following proposition.

\begin{proposition}\label{prop:graef}
	Let $K(x,y) = - \|x - y\|$. 
	For every compact set $\X \subset \mathbb R^d$, there exists a constant $C > 0$ such that the function
	$$
	\tilde K(x,y) \coloneqq C - \|x - y\|
	$$
	is positive definite on $\X$. 
	Further, for $\mu, \nu \in \mathcal{P}(\X)$, it holds
	$$
	\mathscr{D}_{\tilde K}^2(\mu,\nu) = \mathscr{D}_K^2(\mu,\nu) 
	\quad \mathrm{and} \quad 
	\hat \varphi_{\tilde K} = \hat \varphi_{K}.
	$$
\end{proposition}

\begin{proof}
	In \cite[Cor.~2.15]{Graf:2013zl} it was shown that $\tilde K$ is positive definite.
	The rest follows in a straightforward way from \eqref{mercer_1_OT} and \eqref{dual_kern} regarding that $\mu$ and $\nu$ are probability measures.
\end{proof}

Some interesting functions such as negative powers of Euclidean distances or the smoothed distance function
$\sqrt{c^2 + \|x-y\|^2}$, $0 < c \ll 1$, are conditionally positive definite. 
Let $\Pi_{m-1} (\mathbb R^d)$ denote the \smash{$\binom{d + m -1}{d}$}-dimensional space of polynomials on $\mathbb R^d$ of absolute degree (sum of exponents) $\le m-1$.
A function $K\colon \X \times \X \rightarrow \mathbb R$
is \emph{conditionally positive definite of order} $m$ if for 
all points
$x_1, \ldots, x_n \in \mathbb R^d$, $n \in \mathbb N$, the relation
\begin{equation} \label{cpd}
\sum_{i,j=1}^n a_i a_j K(x_i,x_j) \ge 0
\end{equation}
holds true for  all
$a_1, \ldots,a_n \in \mathbb R$
satisfying
\begin{equation}
\sum_{i=1}^n a_i P(x_i) = 0 \qquad \mathrm{for \; all} \quad P \in \Pi_{m-1}(\mathbb R^d).
\end{equation}
If strong inequality holds in \eqref{cpd} except for $a_i=0$ for all $i=1,\ldots,n$, then
$K$ is called \emph{strictly conditionally positive definite of order} $m$.
In particular, for $m=1$, the condition~\eqref{cpd} relaxes to $\sum_{i=1}^n a_i = 0$.

The radial kernels related to the following functions are strictly conditionally positive definite of order $m$ on $\mathbb R^d$:
\begin{align*}
h(r) &= (-1)^{\lceil p\rceil} (c^2 +r^2)^p , && p>0, p \not \in \mathbb N, m =\lceil p \rceil,
\\
h(r) &= (-1)^{\lceil p/2 \rceil} r^p ,
&&p>0, p \not \in 2 \mathbb N, m = \lceil p/2 \rceil,
\\
h(r) &=  (-1)^{k+1}r^{2k} \log(r), &&k \in \mathbb N, m = k+1,
\end{align*}
where $\lceil a \rceil$, denotes the smallest integer larger or equal than $a \in \R$.
The first group of functions are called multiquadric and the last group is known as thin plate splines. 
In connection with Wasserstein distances, the second group of functions is of interest.

By the following lemma, it is easy to turn conditionally positive definite functions into
positive definite ones. However, only for conditionally positive definite functions of order $m=1$,
the discrepancy remains the same.

\begin{lemma} \label{lem:cpd_2}
	Let $\Xi \coloneqq \{u_k: k=1,\ldots,N\}$ with \smash{$N \coloneqq \binom{d + m -1}{m-1}$}
	be a set of points such that 
	$P(u_k) = 0$ for all $k=1,\ldots,N$, $P \in \Pi_{m-1} (\mathbb R^d)$, 
	is only fulfilled for the zero polynomial. 
	Denote by $\{P_k:k=1,\ldots,N\}$ the set of Lagrangian basis polynomials with respect to $\Xi$,
	i.e., $P_k(u_j) = \delta_{jk}$.
	Let $K \in C(\X \times \X)$ be a symmetric conditionally positive definite kernel of order $m$. 
	\begin{itemize}
		\item[i)]
		Then
		\begin{align*}
		\tilde K(x,y) \coloneqq K(x,y) - &\sum_{j=1}^N P_j(x) K(u_j,y) - \sum_{k=1}^N P_k(y) K(x,u_k)\\ 
		+ &\sum_{j,k=1}^N P_j(x)P_k(y) K(u_j,u_k)
		\end{align*}
		is a positive definite kernel.
		\item[ii)] 
		If $\mu$ and $\nu$ have the same moments up to order $m-1$, then they satisfy $\mathscr{D}_{\tilde K}^2(\mu,\nu) = \mathscr{D}_K^2(\mu,\nu)$.
		\item[iii)]
		In particular, we have for $m=1$, $\mu, \nu \in \mathcal{P}(\X)$ and any fixed $u \in \X$ that
		\begin{equation} \label{cpd_1}
		\tilde K(x,y) = K(x,y) -  K(u,y) - K(x,u) + K(u,u)
		\end{equation}
		and
		\begin{align}
		\mathscr{D}_{\tilde K}^2(\mu,\nu) &= \mathscr{D}_K^2(\mu,\nu) , \\
		\hat \varphi_{\tilde K} &= \frac{\int_\X  K(x,\cdot) \dx \mu(x) - \int_\X  K(x,\cdot) \dx \nu(x) + c_\nu - c_\mu}
		{\Vert \int_\X  K(x,\cdot) \dx \mu(x) - \int_\X  K(x,\cdot) \dx \nu(x) + c_\nu - c_\mu \Vert_{H_K(\X)}},
		\end{align}
		where
		\begin{equation} \label{eq:xx}
		c_\mu \coloneqq \int\limits_{\X}  K(x,u) \dx\mu(x) 
		\quad \mathrm{and} \quad
		c_\nu \coloneqq \int\limits_{\X}  K(x,u) \dx\nu(x).
		\end{equation}
	\end{itemize}
\end{lemma}

\begin{proof}
	i) This part follows by straightforward computation, see \cite[Thm.~10.18]{Wendland:2004wd}.
	\\
	ii) Assuming that $\mu$ and $\nu$ have that same moments up to order $m-1$,
	i.e.,
	$$
	p_j = \int_{\X} P_j(x) \dx \mu(x) = \int_{\X} P_j(x) \dx \nu(x), \quad j=1,\ldots,N,
	$$
	and abbreviating for the symmetric kernels
	$$
	c_{\mu,j} \coloneqq \int_{\X} K(u_j,y) \dx \mu(y) ,
	\quad
	c_{\nu,j} \coloneqq \int_{\X} K(u_j,y) \dx \nu(y) ,
	$$
	we obtain by definition of $\tilde K$ that
	\begin{align*}
	&\mathscr{D}_{\tilde K}^2(\mu,\nu)\\ 
	=& 
	\int_{\X^2} \tilde K \dx(\mu \otimes \mu) + 
	\int_{\X^2} \tilde K \dx(\nu \otimes \nu) - 
	2 \int_{\X^2} \tilde K \dx(\mu \otimes \nu)\\
	=&\mathscr{D}_{\tilde K}^2(\mu,\nu) 
	- \sum_{j=1}^N p_j(c_{\mu,j} + c_{\nu,j}) - \sum_{k=1}^N p_j(c_{\mu,k} + c_{\nu,k})
	+ 2 \sum_{j,k=1}^N p_j p_k K(u_j,u_k) \\
	&\quad +
	\sum_{j=1}^N p_j(c_{\mu,j} + c_{\nu,j}) + \sum_{k=1}^N p_j(c_{\mu,k} + c_{\nu,k}) 
	- 2 
	\sum_{j,k=1}^N p_j p_k K(u_j,u_k)\\
	=&
	\mathscr{D}_{K}^2(\mu,\nu). 
	\end{align*}
	iii) Let $m=1$. 
	Then we have for the optimal dual potential in \eqref{dual_kern} related to $\mathscr{D}_{\tilde K}$ that
	\begin{align}
	\hat \varphi_{\tilde K} 
	&=  \frac{\int_\X \tilde K(x,\cdot) \dx \mu(x) - \int_\X \tilde K(x,\cdot) \dx \nu(x)}{\Vert \int_\X \tilde K(x,\cdot) \dx \mu(x) 
		- \int_\X \tilde K(x,\cdot) \dx \nu(x)\Vert_{H_K(\X)}} \\
	&=
	\frac{\int_\X  K(x,\cdot) \dx \mu(x) - \int_\X  K(x,\cdot) \dx \nu(x) + c_\nu - c_\mu}
	{\Vert \int_\X  K(x,\cdot) \dx \mu(x) - \int_\X  K(x,\cdot) \dx \nu(x) + c_\nu - c_\mu\Vert_{H_K(\X)}}.
	\end{align}
\end{proof}

\section{Optimal transport and Wasserstein distances}\label{sec:OT}
The following discussion about optimal transport is based on \cite{Ambrosio,CP2019,S2015}, 
where many aspects simplify due to the compactness of $\X$ and the assumption that the cost $c$ is Lipschitz continuous.
Let $\mu, \nu \in \mathcal P(\X)$ and $c \in C(\X \times \X)$ be a non-negative, symmetric and Lipschitz continuous function. 
Then, the \emph{Kantorovich problem of optimal transport} (OT) reads
\begin{equation}\label{Monge_Kantorovich_problem}
\OT(\mu,\nu) \coloneqq \inf_{\pi \in\Pi(\mu,\nu)} \int_{\X^2} c\, \dx \pi,
\end{equation}
where $\Pi(\mu,\nu)$ denotes the set of joint probability measures $\pi$ on $\X^2$ with marginals $\mu$ and $\nu$.
In our setting, the OT functional \smash{$\pi \mapsto  \int_{\X^2} c\, \dx \pi$} is weakly continuous, 
\eqref{Monge_Kantorovich_problem} has a solution and every such minimizer $\hat \pi$ is called optimal transport plan.
In general, we can not expect the optimal transport plan to be unique.
However, if $\X$ is a compact subset of a separable Hilbert space, $c(x,y) = \Vert x-y \Vert_\X^p$, $p\in (1,\infty)$, 
and either $\mu$ or $\nu$ is \emph{regular}, see \cite[Def.~6.2.2]{Ambrosio} for the technical definition, then \eqref{Monge_Kantorovich_problem} has a unique solution.
Instead of giving the exact definition, we want to remark that for $\X = \R^d$ the regular measures are precisely the ones which have a density with respect to the Lebesgue measure.

The \emph{$c$-transform} $\varphi^{c} \in C(\X)$ of $\varphi \in C(\X)$ is defined as 
\begin{equation*}
\varphi^{c}(y) = \min_{x\in\X}\bigl\{c(x,y)-\varphi(x)\bigr\}.
\end{equation*}
Note that $\varphi^{c}$ has the same Lipschitz constant as $c$.
A function $\varphi^{c} \in C(\X)$ is called $c$-concave if it is the $c$-transform of some function $\varphi \in C(\X)$.

The dual formulation of the OT problem \eqref{Monge_Kantorovich_problem} reads
\begin{equation}\label{Wdual}
\OT(\mu,\nu) = \max_{ \substack{(\varphi,\psi)\in C(\X)^2\\
		\varphi(x)+\psi(y)\leq c(x,y)}}
\int_{\X}\varphi \dx \mu + \int_{\X} \psi \dx \nu.
\end{equation}
Maximizing pairs are essentially of the form $(\varphi,\psi) = (\hat \varphi, \hat \varphi^c)$ for some $c$-concave function $\hat \varphi$ and fulfill $\hat \varphi(x) + \hat \varphi^c(y) = c(x,y)$ in $\supp(\hat \pi)$, where $\hat \pi$ is any optimal transport plan.
The function $\hat \varphi$ is called (Kantorovich) potential for the couple $(\mu,\nu)$.
If \smash{$(\hat \varphi,\hat \psi)$} is an optimal pair, clearly also $(\hat \varphi -C,\hat \psi +C)$ with $C \in \R$ is optimal 
and manipulations outside of $\supp(\mu)$ and $\supp(\nu)$ do not change the functional value.
But even if we exclude such manipulations, the optimal dual potentials are in general not unique as Example~\ref{Ex:1} shows.

\begin{example}\label{Ex:1}
	We choose $\X = [0,1]$, $c(x,y) = \vert x-y \vert$, $\mu = \delta_{0}/2 + \delta_{1}/2$ and $\nu = \delta_{0.1}/2 + \delta_{0.9}/2$.
	Then, $\OT(\mu,\nu) = 0.1$ with the unique optimal transport plan $\hat \pi = \frac12 \delta_{0,0.1} + \frac12 \delta_{1,0.9}$. 
	Optimal dual potentials are  given by
	\[\hat \varphi_1(x) = \begin{cases}
	0.1 - x & \text{ for } x \in [0, 0.1],\\
	x - 0.9 & \text{ for } x \in [0.9, 1],\\
	0 & \text{ else,}
	\end{cases}
	\quad \text { and } \quad
	\hat \varphi_2(x) = \begin{cases}
	0.2 - x & \text{ for } x \in [0, 0.2],\\
	x - 0.9 & \text{ for } x \in [0.9, 1],\\
	0 & \text{ else.}
	\end{cases}\] 
	Clearly, these potentials do not differ only by a constant.
\end{example}

\begin{remark}\label{rem:ExtW}
	Note that the space $C(\X)^2$ in the dual problem could also be replaced with $C(\supp(\mu)) \times C(\supp(\nu))$.
	Using the Tietze extension theorem, any feasible point of the restricted problem can 
	be extended to a feasible point of the original problem and hence the problems coincide.
	If the problem is restricted, all other concepts have to be adapted accordingly.
\end{remark}

For $p \in [1,\infty)$,
the \emph{$p$-Wasserstein distance}  $W_p$ between $\mu,\nu \in \mathcal P(\X)$ is defined by
\begin{align} \label{eq:OTprimal}
W_p(\mu,\nu) \coloneqq \biggl( \min_{\pi \in \Pi(\mu,\nu)} 
\int_{\X^2} \dist(x,y)^p \mathrm{d} \pi(x,y) \biggr)^\frac{1}{p}.
\end{align}
It is a metric on $\mathcal P (\X)$, which metrizes the weak topology.
Indeed, due to compactness of $\X$, we have that $\mu_k \weakly \mu$
if and only if $\lim_{k \rightarrow \infty} W_p (\mu_k, \mu) = 0$.

For $1 \le p \le q < \infty$ it holds $W_p \le W_q$.
The distance $W_1$ is also called
\emph{Kantorovich-Rubinstein distance} or \emph{Earth's mover distance}.
Here, it holds $\varphi^c = -\varphi$ and the dual problem reads
\begin{equation} \label{wasser_1_dual}
W_1(\mu,\nu)
= \max_{|\varphi|_{\Lip(\X)} \le 1}  \int_{\X} \varphi \dx \xi ,\quad \xi \coloneqq \mu -  \nu,
\end{equation}
where the maximum is taken over all Lipschitz continuous functions with Lipschitz constant bounded by 1.
This looks similar to the discrepancy \eqref{equiv_1_OT}, but the space of test functions is larger for $W_1$.

The distance $W_1$ is related to $W_p$ by
$$
W_1(\mu,\nu) \le W_p(\mu,\nu) \le C W_1(\mu,\nu)^\frac{1}{p}
$$
with a constant $0 \le C < \infty$ depending on $\diam(\X)$ and $p$. 


\section{Regularized optimal transport} \label{sec:sinkhorn}
In this section, we give a self-contained introduction to continuous \emph{regularized optimal transport}.
For $ \mu, \nu \in {\mathcal P}(\X)$ and $\varepsilon > 0$, regularized OT is defined as
\begin{equation} \label{sinkhorn} 
\OT_{\varepsilon}(\mu,\nu) 
\coloneqq  
\min_{\pi \in \Pi(\mu,\nu)} 
\,
\Big\{
\int_{\X^2}  c\, \mathrm{d} \pi
+ \varepsilon \mathrm{KL} (\pi,\mu \otimes \nu) \Big\}.
\end{equation}
Compared to the original $\OT$ problem, we will see in the numerical part that $\OT_{\varepsilon}$ can be  efficiently solved numerically, see also~\cite{CP2019}.
Moreover, $\OT_{\varepsilon}$ has the following properties.

\begin{lemma}\label{lem:1}
	\begin{itemize}
		\item[i)] There is a unique minimizer $\hat \pi_\varepsilon \in  {\mathcal P}(\X^2)$ of \eqref{sinkhorn} with finite value.
		\item[ii)] The function $\OT_\varepsilon$ is weakly continuous and Fr\'echet differentiable.
		\item[iii)] For any $\mu, \nu \in \mathcal P(\X)$ and $\varepsilon_1, \varepsilon_2 \in [0,\infty]$ with  $\varepsilon_1 \leq \varepsilon_2$ it holds
		\[\OT_{\varepsilon_1}(\mu,\nu) \leq \OT_{\varepsilon_2}(\mu,\nu).\]
	\end{itemize}
\end{lemma}

\begin{proof}
	i) First, note that $\mu \otimes \nu$ is a feasible point and hence the infimum is finite.
	Existence of minimizers follows as the functional is weakly lsc and $\Pi(\mu,\nu) \subset \mathcal P(\X^2)$ is weakly compact.
	Uniqueness follows since $\mathrm{KL}(\cdot, \mu \otimes\nu)$ is strictly convex.
	
	ii) The proof uses the dual formulation in Proposition \ref{prop:dual}, see~\cite[Prop.~2]{FSVATP2018}.
	
	iii) Let $\hat \pi_{\varepsilon_2}$ be the minimizer for $\OT_{\varepsilon_2}(\mu,\nu)$.
	Then, it holds
	\begin{align*}
	\OT_{\epsilon_2}(\mu,\nu) &= \int_{\X^2}  c\, \dx \hat \pi_{\varepsilon_2}
	+ \varepsilon_2 \mathrm{KL} (\hat \pi_{\varepsilon_2},\mu \otimes \nu)\\
	&\geq \int_{\X^2}  c\, \dx \hat \pi_{\varepsilon_2}
	+ \varepsilon_1 \mathrm{KL} (\hat \pi_{\varepsilon_2},\mu \otimes \nu) \geq \OT_{\epsilon_1}(\mu,\nu).
	\end{align*}
\end{proof}

Note that in special cases, e.g., for absolutely continuous measures, see \cite{CDPS17, Leo12}, 
it is possible to show convergence of the optimal solutions $\hat \pi_\varepsilon$ to an optimal solution of $\OT(\mu,\nu)$ as $\varepsilon \to 0$.
However, we are not aware of a fully general result.
An extension of entropy regularization to unbalanced $\OT$ is discussed in \cite{CPSV17}.

Originally, entropic regularization was proposed in \cite{C2013} for \emph{discrete} probability measures with the negative entropy $E$, see also \cite{Peyre2015},
\[\widetilde{\OT}_\varepsilon (\mu,\nu) \coloneqq \min_{\pi \in \Pi(\mu,\nu)} \Bigl\{ \int_{\X^2}  c\, \mathrm{d} \pi + \varepsilon E(\pi) \Bigr\}, 
\quad 
E(\pi) \coloneqq \sum_{i,j=1}^n  \log( p_{ij}) p_{ij} = \mathrm{KL} (\pi,\lambda \otimes \lambda),
\]
where $\lambda$ denotes the counting measure.
For $\pi \in \Pi(\mu,\nu)$ it is easy to check that
$$E(\pi) = \mathrm{KL}(\pi,\mu\otimes \nu) + \sum_{i,j = 1}^n \log(\mu_i \nu_j) \mu_i \nu_j
= \mathrm{KL}(\pi,\mu\otimes \nu) +\mathrm{KL}(\mu\otimes \nu,\lambda \otimes \lambda),
$$
i.e., the minimizers are independent of the chosen regularization.
For non-discrete measures, special care is necessary as the following remark shows.

\begin{remark}\upshape\textbf{$\mathrm{(KL}(\pi,\mu\otimes\nu)$ versus  $E(\pi)$ regularization$\mathrm )$} \label{rem:KL-versus-entropy}
	Since the entropy is only defined for measures with densities, we consider compact sets $\X \subset \R^d$ equipped with the normalized Lebesgue measure $\lambda$ and $\mu,\nu \ll \lambda$ with densities $\sigma_\mu, \sigma_\nu \in L^1(\X)$.
	For $\pi \ll \lambda \otimes \lambda$ with density $\sigma_\pi$ the entropy is defined by
	$$
	E(\pi) = \int_{\X^2} \log(\sigma_\pi) \, \sigma_\pi \dx (\lambda \otimes \lambda) = \mathrm{KL}(\pi,\lambda \otimes \lambda).
	$$
	Note that for any $\pi \in \Pi(\mu,\nu)$ we have
	\[ \pi \ll \mu \otimes \nu  \quad \Longleftrightarrow \quad \pi \ll \lambda \otimes \lambda,\]
	where the right implication follows directly and the left one can be seen as follows:
	If $\pi \ll \lambda \otimes \lambda$ with density $\sigma_\pi \in L^1(\X \times \X)$, then
	\[0 = \int_{\{z \in \X: \sigma_\mu(z) = 0\}} \int_{\X} \sigma_\pi(x,y) \dx y \dx x.\]
	Consequently, we get $\sigma_\pi(x,y) = 0$ a.e.~on $\{z \in \X: \sigma_\mu(z) = 0\} \times \X$ (for any representative of $\sigma_\mu$).
	The same reasoning is applicable to $\X \times \{z \in \X: \sigma_\nu(z) = 0\}$.
	Thus,
	\[\pi = \sigma_\pi \, (\lambda\otimes\lambda) = \frac{\sigma_\pi(x,y)}{\sigma_\mu(x) \sigma_\nu(y)} \, (\mu \otimes \nu),\]
	where the quotient is defined as zero if $\sigma_\mu$ or $\sigma_\nu$ vanish.
	Hence, the left implication also holds true.
	
	If $\mathrm{KL} (\mu \otimes \nu, \lambda \otimes \lambda) < \infty$, we conclude for any $\pi \ll \lambda \otimes \lambda$ with $\pi \in \Pi(\mu,\nu)$ that the following expressions are well-defined 
	\begin{align*}
	&\mathrm{KL} (\pi,\lambda \otimes \lambda) - \mathrm{KL} (\mu \otimes \nu, \lambda \otimes \lambda) \label{care}\\
	&=\int_{\X^2} \log(\sigma_\pi) \, \dx \pi - \int_{\X^2} \log\Bigl(\frac{\dx(\mu \otimes \nu)}{\dx (\lambda \otimes \lambda)} \Bigr) \, \dx (\mu \otimes \nu) \\
	&=\mathrm{KL} (\pi,\mu \otimes \nu) + \int_{\X^2} \log\bigl(\sigma_\mu(x) \sigma_\nu(y) \bigr) \, \dx \pi(x,y) 
	- \int_{\X^2} \log\bigl(\sigma_\mu(x) \sigma_\nu(y) \bigr) \, \dx \mu(x) \dx \nu(y)\\
	&=\mathrm{KL} (\pi,\mu \otimes \nu).
	\end{align*}
	Consequently, in this case we also have  \smash{$\widetilde{\OT}_\varepsilon (\mu,\nu) ={\OT}_\varepsilon (\mu,\nu) + \varepsilon \mathrm{KL}(\mu\otimes \nu,\lambda \otimes \lambda)$}.
	The crux is the condition $\mathrm{KL} (\mu \otimes \nu, \lambda \otimes \lambda) < \infty$, which is equivalent to $\mu,\nu$ having finite entropy, 
	i.e.,  $\sigma_\mu, \sigma_\nu$ are in a so-called Orlicz space $L \log L$ \cite{NR2013}.
	The authors in \cite{CLMW19} considered the entropy as regularization (with continuous cost function) and
	pointed out that \smash{$\widetilde{\OT}_\varepsilon (\mu,\nu)$} admits a (finite) minimizer exactly in this case.
	However, we have seen that we can avoid this existence trouble if we regularize with $\mathrm{KL} (\pi,\mu \otimes \nu)$ instead, 
	which therefore seems to be a more natural choice.
	A comparison of the settings and a more general existence discussion based on merely continuous cost functions can be also found in \cite{MG2019}.	
\end{remark}

Another possibility is to use quadratic regularization instead, see \cite{LMM19} for more details.
In connection with discrepancies, we are especially interested in the limiting case $\varepsilon \rightarrow \infty$. 
The next proposition is basically known, see \cite{CP2019,FSVATP2018}.
However, we have not found it in this generality in the literature.

\begin{proposition}\label{prop:conv}
	\begin{itemize}
		\item[i)]	It holds $\lim_{\varepsilon \to \infty} \OT_\varepsilon(\mu,\nu) = \OT_\infty(\mu,\nu)$, 
		where
		\[\OT_\infty(\mu,\nu) \coloneqq \int_{\X^2} c\, \dx (\mu \otimes \nu).\]
		\item[ii)]
		It holds $\lim_{\varepsilon \to 0} \OT_\varepsilon(\mu,\nu) = \OT(\mu,\nu)$.
	\end{itemize}
\end{proposition}

\begin{proof}
	i) For $\pi = \mu \otimes \nu$, we have 
	\[
	\int_{\X^2}  c\,\dx \pi 
	+ \varepsilon \mathrm{KL} (\pi,\mu \otimes \nu)  = \OT_\infty(\mu,\nu)
	\]
	and consequently 
	$\limsup_{\varepsilon \to \infty} \OT_\varepsilon(\mu,\nu) \leq \OT_\infty(\mu,\nu)$.
	In particular, the optimal transport plan $\hat \pi_\varepsilon$ satisfies
	$\limsup_{\varepsilon \to \infty} \varepsilon \text{KL}(\hat \pi_\varepsilon ,  \mu \otimes \nu) \leq \OT_\infty(\mu,\nu)$.
	Since $\text{KL}$ 
	is weakly lsc, 
	we conclude that the sequence of minimizers  $\hat \pi_\varepsilon$ satisfies $\hat \pi_\varepsilon \weakly \mu \otimes \nu$ as $\varepsilon \to \infty$.
	Hence, we obtain the desired result from
	\begin{align*}
	\liminf_{\varepsilon \to \infty} \OT_\varepsilon(\mu,\nu) 
	&= 
	\liminf_{\varepsilon \to \infty} \int_{\X^2}  c\,\dx \hat \pi_\varepsilon + \varepsilon \mathrm{KL} (\hat \pi_\varepsilon,\mu \otimes \nu)\\
	&\geq  
	\liminf_{\varepsilon \to \infty} \int_{\X^2}  c\,\dx \hat \pi_\varepsilon  = \OT_\infty(\mu,\nu).
	\end{align*}
	\noindent
	
	ii) This part is more involved and follows from of Proposition~\ref{lem:PropPot} ii).
\end{proof}

Similar as $\OT$ in \eqref{Wdual}, its regularized version $\OT_\varepsilon$ can be written in dual form, see~\cite{CPSV17,CLMW19}.

\begin{proposition}\label{prop:dual}
	The (pre-)dual problem of $\OT_\varepsilon$ is given by
	\begin{align}
	\OT_\varepsilon(\mu,\nu) 
	&= \sup_{(\varphi,\psi)\in C(\X)^2}
	\Big\{ \int_{\X}\varphi \dx \mu + \int_{\X} \psi \dx \nu \\
	& \qquad - \varepsilon \int_{\X^2} \exp\Bigl(\frac{\varphi(x) + \psi(y) - c(x,y)}{\varepsilon}\Bigr) -1
	\dx (\mu \otimes \nu) \Big\}.\label{pre-dual}
	\end{align}
	If optimal dual solutions $\hat \varphi_\varepsilon$ and $\hat \psi_\varepsilon$ exist, 
	they are related to the optimal transport plan $\hat \pi_\varepsilon$ by
	\begin{equation}\label{eq:PDrelation}
	\hat\pi_\varepsilon = \exp\Bigl(\frac{\hat \varphi_\varepsilon(x) + \hat \psi_\varepsilon(y) - c(x,y)}{\varepsilon}\Bigr) \mu \otimes \nu.
	\end{equation}
\end{proposition}

\begin{proof} 
	Let us consider
	$F \in \Gamma_0(C(\X)^2)$, $G \in \Gamma_0(C(\X^2))$ with Fenchel conjugates  $F^* \in \Gamma _0 (\mathcal M(\X)^2)$, $G^* \in \Gamma_0 (\mathcal M(\X^2))$
	together with a linear bounded operator $A\colon C(\X)^2 \to C(\X^2)$ with adjoint operator $A^* \colon\mathcal M(\X^2) \to \mathcal M(\X)^2$
	defined by
	\begin{align}
	&F(\varphi,\psi) = \int_{\X}\varphi \dx \mu + \int_{\X} \psi \dx \nu,\\
	&G(\varphi) = \varepsilon \int_{\X^2} \exp\Bigl(\frac{\varphi - c}{\varepsilon}\Bigr) -1 \dx (\mu \otimes \nu),\\
	&A(\varphi,\psi)(x,y) =  \varphi(x) + \psi(y).
	\end{align}
	Then, \eqref{pre-dual} has the form of the left-hand side in \eqref{primal-dual}.
	Incorporating \eqref{dual_G}, we get
	$$
	G^*(\pi) = \int_\X c \dx  \pi + \varepsilon \mathrm{KL}( \pi,\mu\otimes\nu).
	$$
	Using the indicator function $\iota_C$ defined by $\iota_C (x)\coloneqq 0$ for $x \in C$ and $\iota_C (x)\coloneqq +\infty$ otherwise, we have 
	\begin{align*}
	F^*(A^*\pi) 
	&= \sup_{(\varphi,\psi) \in C(\X)^2} \langle A^* \pi,(\phi,\psi) \rangle -  \int_{\X}\varphi \dx \mu - \int_{\X} \psi \dx \nu\\
	&= \sup_{(\varphi,\psi) \in C(\X)^2} \langle  \pi, \phi(x) + \psi(y) \rangle -  \int_{\X}\varphi \dx \mu - \int_{\X} \psi \dx \nu \\
	&= \iota_{\Pi(\mu,\nu)} (\pi).
	\end{align*}
	Now, the duality relation follows from \eqref{primal-dual}. 
	
	If the optimal solution \smash{$(\hat \varphi_\varepsilon,\hat \psi_\varepsilon)$} exists, we can apply \eqref{dual_2} and \eqref{eq:GradKL}
	to obtain
	$$
	\hat\varphi_\varepsilon(x) + \hat \psi_\varepsilon(y) = c + \log\left( \frac{\dx \hat \pi_\varepsilon}{\dx (\mu\otimes \nu) } \right),
	$$
	which  yields \eqref{eq:PDrelation}.
\end{proof}

\begin{remark}\label{rem:restrict}
	Using the Tietze extension theorem, we could also replace the space $C(\X)^2$ by $C(\supp(\mu)) \times C(\supp(\nu))$.
\end{remark}

Note that the last term in \eqref{pre-dual} is a smoothed version of the associated constraint $\varphi(x)+\psi(y)\leq c(x,y)$ appearing in~\eqref{Wdual}.
Clearly, the values of $\varphi$ and $\psi$  are only relevant on $\supp(\mu)$ and $\supp(\nu)$, respectively.
Further, for any $\varphi, \psi \in C(\X)$ and $C \in \R$, the potentials $\varphi + C, \psi -C$ realize the same value in \eqref{pre-dual}.

For fixed $\varphi$ or $\psi$, the corresponding maximizing potentials in \eqref{pre-dual} are given by 
$$
\hat \psi_{\varphi,\varepsilon} = T_{\mu,\varepsilon}(\varphi) \text{ on $\supp(\nu)$} \quad \mathrm{and} \quad \hat \varphi_{\psi,\varepsilon} = T_{\nu,\varepsilon}(\psi) \text{ on $\supp(\mu)$},
$$
respectively.
Here, $T_{\mu,\varepsilon} \colon C(\X) \to C(\X)$ is defined as
\begin{align}
T_{\mu,\varepsilon}(\varphi)(x) \coloneqq - \varepsilon \log \left( \int_\X \exp\Bigl(\frac{\varphi(y) -c(x,y)}{\varepsilon}\Bigr)\dx \mu(y) \right).\label{eq:condphi}
\end{align}
Therefore, any pair of optimal potentials $\hat \varphi_\varepsilon$ and $\hat \psi_\varepsilon$ must satisfy
\begin{align}\label{eq:CondOptim}
\hat \psi_\varepsilon  = T_{\mu,\varepsilon}(\hat \varphi_\varepsilon) \text{ on $\supp(\nu)$}, \qquad \hat \varphi_\varepsilon = T_{\nu,\varepsilon}(\hat \psi_\varepsilon) \text{ on $\supp(\mu)$}.
\end{align}
For every $\varphi \in C(\X)$ and $C \in \R$, it holds $T_{\mu,\varepsilon}(\varphi+C) = T_{\mu,\varepsilon}(\varphi) +C$.
Hence, $T_{\mu,\varepsilon}$ can be interpreted as an operator on the quotient space $C(\X)/\R$, where $f_1,f_2 \in C(\X)$ are equivalent if they differ by a real constant.
This space can equipped with the \emph{oscillation norm}
\[\Vert f \Vert_{\circ,\infty} \coloneqq \tfrac12(\max f - \min f)\]
and for $ f \in C(\X)/\R$ there is a representative $\bar f \in C(\X)$ with $\Vert f \Vert_{\circ,\infty} = \Vert \bar f \Vert_{\infty}$.
Finally, it is possible to restrict the domain of $T_{\mu,\varepsilon}$ to $C(\supp(\mu))$ and $C(\supp(\mu))/\R$, respectively.
This interpretation is useful for showing convergence of the Sinkhorn algorithm.
In the next lemma, we collect a few properties of $T_{\mu,\varepsilon}$, see also \cite{GCBCP2018, Via19}.

\begin{lemma}\label{prop:PropT}
	\begin{itemize}
		\item[i)] 
		For any measure $\mu \in P(\X)$, $\varepsilon > 0$ and $\varphi \in C(\X)$, the function $T_{\mu,\varepsilon}(\varphi) \in C(\X)$ has the same Lipschitz constant as $c$ and satisfies
		\begin{equation}\label{eq:rangePsi}
		T_{\mu,\varepsilon}(\varphi)(x) \in \Bigl[\,\,\min_{y \in \supp(\mu)}  c(x,y) - \varphi(y), \max_{y \in \supp(\mu)} c(x,y) - \varphi(y)\Bigr].
		\end{equation}		
		\item[ii)] 
		For fixed $\mu \in \mathcal P(\X)$, the operator $T_{\mu,\varepsilon}\colon C(\supp(\mu)) \to C(\X)$ is $1$-Lipschitz.
		Additionally, the operator $T_{\mu,\varepsilon}\colon C(\supp(\mu))/\R \to C(\X)/\R$ is $\kappa$-Lipschitz with $\kappa < 1$.
	\end{itemize}
\end{lemma}

\begin{proof} i)
	For $x_1,x_2 \in \X$ (possibly changing the naming of the variables) we obtain
	\begin{align}
	&\bigl\vert T_{\mu,\varepsilon}(\varphi)(x_1) - T_{\mu,\varepsilon}(\varphi)(x_2) \bigr\vert\\
	=& \varepsilon  
	\Big| \log  \int_\X \exp\Bigl(\frac{\varphi(y)-c(x_2,y)}{\varepsilon}\Bigr)\dx \mu(y) 
	-  \log  \int_\X \exp\Bigl(\frac{\varphi(y)-c(x_1,y)}{\varepsilon}\Bigr)\dx \mu(y) \Big|
	\\
	=& \varepsilon  \log \left( \int_\X \exp\Bigl(\frac{\varphi(y)-c(x_2,y)}
	{\varepsilon}\Bigr)\dx \mu(y)\Bigr/ \int_\X \exp\Bigl(\frac{\varphi(y)-c(x_1,y)}{\varepsilon}\Bigr)\dx \mu(y) \right).
	\end{align}
	Incorporating the $L$-Lipschitz continuity of $c$, we get
	\begin{align*}
	\exp\Bigl(\frac{c(x_1,y) -c(x_2,y) }{\varepsilon}\Bigr) 
	\le \exp\Bigl(\frac{|c(x_1,y) - c(x_2,y)| }{\varepsilon}\Bigr)
	\le \exp \Bigl(\frac{L}{\varepsilon} |x_1 - x_2| \Bigr),
	\end{align*}
	so that
	\begin{align*}
	\int_\X \exp\Bigl(\frac{\varphi(y) - c(x_2,y)}{\varepsilon} \Bigr) \dx \mu(y)
	&\leq 
	\exp \Bigl(\frac{L}{\varepsilon} |x_1 - x_2|\Bigr) \int_\X \exp\Bigl(\frac{\varphi(y) - c(x_1,y)}{\varepsilon} \Bigr) \dx \mu(y).
	\end{align*}
	Thus, $T_{\mu,\varepsilon}(\varphi)$ is Lipschitz continuous
	\begin{align}
	\bigl\vert T_{\mu,\varepsilon}(\varphi)(x_1) - T_{\mu,\varepsilon}(\varphi)(x_2) \bigr\vert	
	\le 
	\varepsilon  \log \Bigl( \exp\Bigl( \frac{L}{\varepsilon} \vert x_1 -x_2 \vert\Bigr) \Bigr)
	= L \vert x_1 -x_2 \vert.
	\end{align}
	Finally, \eqref{eq:rangePsi} follows directly from \eqref{eq:condphi} since $\mu$ is a probability measure.
	\\
	ii)
	For any $x \in \X$ and $\varphi_1,\varphi_2 \in C(\supp(\mu))$ it holds 
	\begin{align}\label{eq:EstStand}
	T_{\mu,\varepsilon}(\varphi_1)(x) - T_{\mu,\varepsilon}(\varphi_2)(x) =& \int_0^1 \tfrac{\dx}{\dx t} T_{\mu,\varepsilon}\bigl(\varphi_1 + t(\varphi_2 - \varphi_1)\bigr)(x) \dx t\\
	=&\int_0^1 \int_\X \bigl(\varphi_1(z) - \varphi_2(z) \bigr) \rho_{t,x}(z) \dx \mu(z) \dx t\notag
	\end{align}
	with
	\[\rho_{t,x} \coloneqq \frac{\exp\bigl(\bigl(t\varphi_2 + (1-t)\varphi_1 - c(x,\cdot)/\varepsilon\bigr)\bigr)}{\int_\X \exp\bigl(\bigl(t\varphi_2(z) + (1-t)\varphi_1(z) - c(x,z)\bigr)/\varepsilon\bigr) \dx \mu(z)}.\]
	This directly implies
	\[\Vert T_{\mu,\varepsilon}(\varphi_1) - T_{\mu,\varepsilon}(\varphi_2) \Vert_\infty \leq \sup_{ x \in \supp(\mu)} \int_0^1 \int_\X \bigl\vert \varphi_1(z) - \varphi_2(z) \bigr\vert \rho_{t,x}(z) \dx \mu(z) \dx t \leq \Vert \varphi_1 - \varphi_2 \Vert_\infty.\]
	
	In order to show the second claim, we choose representatives $\varphi_1$ and $\varphi_2$ such that $\Vert \varphi_1 - \varphi_2 \Vert_{\infty} = \Vert \varphi_1 - \varphi_2 \Vert_{\circ, \infty}$.
	Given $x,y\in \X$, we conclude using \eqref{eq:EstStand} that
	\begin{align}
	&\frac{1}{2}\bigr(T_{\mu,\varepsilon}(\varphi_1)(x) - T_{\mu,\varepsilon}(\varphi_2)(x) - T_{\mu,\varepsilon}(\varphi_1)(y) + T_{\mu,\varepsilon}(\varphi_2)(y)\bigl)\notag\\
	=& \frac{1}{2} \int_0^1 \int_\X \bigl(\varphi_1(z) - \varphi_2(z) \bigr) \bigr(\rho_{t,x}(z) - \rho_{t,y}(z) \bigl) \dx \mu(z) \dx t\notag\\
	\leq& \Vert \varphi_1 - \varphi_2 \Vert_{\circ,\infty} \frac{1}{2} \int_0^1 \Vert \rho_{t,x} - \rho_{t,y} \Vert_{L^1(\mu)} \dx t.\label{eq:EstExp}
	\end{align}
	For all $z \in \X$ with $p_{t,x}(z) \geq p_{t,y}(z)$, we can estimate
	\[p_{t,x}(z) - p_{t,y}(z) \leq p_{t,x}(z)(1 - \exp(-2L\diam(\X)/\varepsilon))\]
	and similarly for $z \in \X$ with $p_{t,y}(z) \geq p_{t,x}(z)$.
	Hence, we obtain
	\begin{align*}
	\Vert \rho_{t,x} - \rho_{t,y} \Vert_{L^1(\mu)}
	\leq& \int_{\X} (1_{\{p_{t,x}\geq p_{t,y}\}}p_{t,x} + 1_{\{p_{t,y}>p_{t,x}\}}p_{t,y})\bigl(1 - \exp(-2L\diam(\X)/\varepsilon)\bigr)\dx\mu\\
	\leq&2\bigl(1 - \exp(-2L\diam(\X)/\varepsilon)\bigr).
	\end{align*}
	Finally, inserting this into \eqref{eq:EstExp} implies
	\[\bigl\Vert T_{\mu,\varepsilon}(\varphi_1) - T_{\mu,\varepsilon}(\varphi_2) \bigr\Vert_{\circ,\infty} \leq \bigl(1-\exp(-2L\diam(\X)/\varepsilon)\bigr) \Vert \varphi_1 - \varphi_2 \Vert_{\circ, \infty}.\]
\end{proof}

Now, we are able to prove existence of an optimal solution $(\hat \varphi_\varepsilon,\hat \psi_\varepsilon)$.

\begin{proposition}\label{prop:existence}
	The optimal potentials $\hat \varphi_\varepsilon, \hat \psi_\varepsilon \in C(\X)$ exist and are unique on $\supp(\mu)$ and $\supp(\nu)$, respectively (up to the additive constant).
\end{proposition}

\begin{proof}
	Let $\varphi_n, \psi_n \in C(\X) $ be maximizing sequences of \eqref{pre-dual}.
	Using the operator $T_{\mu,\varepsilon}$, these can be replaced by 
	$$\tilde \psi_n = T_{\mu,\varepsilon}(\varphi_n) 
	\quad \mathrm{and} \quad 
	\tilde \varphi_n = T_{\nu,\varepsilon} \circ T_{\mu,\varepsilon}(\varphi_n),
	$$
	which are Lipschitz continuous with the same constant as $c$ by Lemma \ref{prop:PropT} i) and therefore uniformly equi-continuous.
	Next, we can choose some $x_0 \in \supp(\mu)$ and w.l.o.g.~assume $\tilde \psi_n(x_0) = 0$.
	Due to the uniform Lipschitz continuity, the potentials $\tilde \psi_n$ are uniformly bounded and by \eqref{eq:rangePsi} the same holds true for $\tilde \varphi_n$.
	Now, the theorem of Arzelà--Ascoli implies that both sequences contain convergent subsequences.
	Since the functional in \eqref{pre-dual} is continuous, we can readily infer the existence of optimal potentials $\hat \varphi_\varepsilon, \hat \psi_\varepsilon \in C(\X)$.
	Due to the uniqueness of $\hat \pi_\varepsilon$, \eqref{eq:PDrelation} implies that $\hat \varphi_\varepsilon|_{\supp(\mu)}$ and $\hat \psi_\varepsilon|_{\supp(\nu)}$ are uniquely determined up to an additive constant. 
\end{proof}

Combining the optimality condition \eqref{eq:condphi} and \eqref{pre-dual}, we directly obtain for any pair of optimal solutions
\begin{equation}\label{eq:optimal energy}
\OT_\varepsilon(\mu,\nu) = \int_{\X}\hat \varphi_\varepsilon  \dx \mu + \int_{\X} \hat \psi_\varepsilon  \dx \nu.
\end{equation}
Adding, e.g., the additional constraint 
\begin{equation} \label{add_constraint}
\int_{\X } \varphi \dx \mu = \tfrac{1}{2} \OT_\infty(\mu,\nu),
\end{equation}
the restricted optimal potentials $\hat \varphi_\varepsilon|_{\supp(\mu)}$ and $\hat \psi_\varepsilon|_{\supp(\nu)}$ are unique.
The next proposition investigates the limits of the potentials as $\varepsilon \to 0$ and $\varepsilon \to \infty$.

\begin{proposition}\label{lem:PropPot}
	\begin{itemize}
		\item[i)] Enforcing the constraint \eqref{add_constraint},
		the restricted potentials 
		$\hat \varphi_\varepsilon|_{\supp(\mu)}$ and $\hat \psi_\varepsilon|_{\supp(\nu)}$ converge uniformly for $\varepsilon \to \infty$ to
		\begin{align}
		\hat \varphi_\infty(x) &= \int_{\X} c(x,y) \dx \nu(y) -  \tfrac{1}{2}\OT_\infty(\mu,\nu),\\
		\hat \psi_\infty(y) &= \int_{\X} c(x,y) \dx \mu(x) -  \tfrac{1}{2}\OT_\infty(\mu,\nu),
		\end{align}	
		respectively.
		\item[ii)] For $\varepsilon \to 0$ every accumulation point of $(\hat \varphi_\varepsilon|_{\supp(\mu)}$, $\hat \psi_\varepsilon|_{\supp(\nu)})$ 
		can be extended to an optimal dual pair for $\OT(\mu,\nu)$ satisfying \eqref{add_constraint}.
		In particular, $\lim_{\varepsilon \to 0} \OT_\varepsilon (\mu,\nu) = \OT(\mu,\nu)$.
	\end{itemize}
\end{proposition}

\begin{proof}
	i) Since $\X$ is bounded, the Lipschitz continuity of the potentials together with \eqref{add_constraint} implies that all $\hat \varphi_\varepsilon$ are uniformly bounded on $\supp(\mu)$.
	Then, we conclude for $y \in \supp(\nu)$ using l'H\^{o}pital's rule, dominated convergence and \eqref{add_constraint} that
	\begin{align*}
	&\lim_{\varepsilon \to \infty} \hat \psi_{\varepsilon}(y) \\
	=& \lim_{\varepsilon \to \infty} -\frac{\int_\X \bigl(\hat \varphi_{\varepsilon}(x) - c(x,y)\bigr) 
		\exp\bigl(\bigl(\hat \varphi_{\varepsilon}(x)-c(x,y)\bigr)/ \varepsilon \bigr)\dx \mu(x)}{\int_\X \exp\bigl(\bigl(\hat \varphi_{\varepsilon}(x)-c(x,y)\bigr) / \varepsilon \bigr)\dx \mu(x)}
	\\
	=& 
	\lim_{\varepsilon \to \infty} \int_\X c(x,y) \exp\bigl(\bigl(\hat \varphi_{\varepsilon}(x)-c(x,y)\bigr)/ \varepsilon \bigr) 
	- \hat \varphi_{\varepsilon}(x) \exp\bigl(\bigl(\hat \varphi_{\varepsilon}(x)-c(x,y)\bigr)/ \varepsilon \bigr) \dx \mu(x)\\
	=& 
	\int_\X c(x,y) \dx \mu(x)
	- \lim_{\varepsilon \to \infty}\int_\X \hat \varphi_{\varepsilon}(x) \Bigl(\exp\bigl(\bigl(\hat \varphi_{\varepsilon}(x)-c(x,y)\bigr)/ \varepsilon\bigr) - 1\Bigr) + \hat \varphi_{\varepsilon}(x) \dx \mu(x)\\
	=& 
	\int_\X c(x,y) \dx \mu(x) -  \tfrac{1}{2}\OT_\infty(\mu,\nu).
	\end{align*}
	Again, a similar reasoning, incorporating \eqref{eq:rangePsi},  can be applied for $\hat \varphi_\varepsilon$.
	Finally, note that pointwise convergence of uniformly Lipschitz continuous functions on compact sets implies uniform convergence.
	\\[1ex]	
	ii) By continuity of the integral, we can directly infer that \eqref{add_constraint} is satisfied for any accumulation point.
	Note that for any fixed $\varphi \in C(\X)$, $x \in \X$ and $\varepsilon \to 0$ it holds
	\begin{align}
	T_{\mu,\varepsilon}(\varphi)(x) \to \min_{y \in \supp(\mu)} c(x,y) - \varphi(y),
	\end{align}
	see~\cite[Prop.~9]{FSVATP2018}, which by uniform Lipschitz continuity of $T_{\mu,\varepsilon}(\varphi)$ directly implies the convergence in $C(\X)$.
	Let \smash{$\{(\hat \varphi_{\varepsilon_j}, \hat \psi_{\varepsilon_j})\}_j$} be a subsequence converging to $(\hat \varphi_0, \hat \psi_0) \in C(\supp(\mu)) \times C(\supp(\nu))$.
	Then, we have
	\begin{align}
	\hat \psi_0 &= \lim_{j \rightarrow \infty} \hat \psi_{\varepsilon_j} = \lim_{j \rightarrow \infty} T_{\mu,\varepsilon_j} (\hat\varphi_{\varepsilon_j})\\
	&=\lim_{j \rightarrow \infty} \left( T_{\mu,\varepsilon_j} (\hat\varphi_{\varepsilon_j}) - T_{\mu,\varepsilon_j}(\hat \varphi_0) +  T_{\mu,\varepsilon_j}(\hat \varphi_0) \right).
	\end{align}
	By Lemma~\ref{prop:PropT} ii), it holds
	$$
	\|T_{\mu,\varepsilon_j} (\hat\varphi_{\varepsilon_j}) - T_{\mu,\varepsilon_j}(\hat \varphi_0)\|_\infty 
	\le \| \hat\varphi_{\varepsilon_j} - \hat \varphi_0\|_\infty
	$$
	and we conclude
	$$
	\hat \psi_0 = \lim_{j \rightarrow \infty}  T_{\mu,\varepsilon_j}(\hat \varphi_0) = \min_{y \in \supp(\mu)} c(\cdot,y) - \hat \varphi_0(y).
	$$
	Similarly, we get 
	\begin{align}
	\hat \varphi_0 &=  \min_{y \in \supp(\nu)} c(\cdot,y) - \hat \psi_0(y).
	\end{align}
	Thus, $(\hat \varphi_0,\hat \psi_0)$ can be extended to a feasible point in $C(\X)^2$ of \eqref{Wdual} by Remark~\ref{rem:ExtW}.
	
	Due to continuity of \eqref{eq:optimal energy} and since $\mathrm{OT}_\varepsilon$ is monotone in $\varepsilon$, 
	this implies 
	\[\lim_{j \to \infty} \OT_{\varepsilon_j}(\mu,\nu) = \int_\X \hat \varphi_0 \dx \mu + \int_\X \hat \psi_0 \dx \nu
	\leq  
	\OT(\mu,\nu) \leq \lim_{j \to \infty} \OT_{\varepsilon_j} (\mu,\nu).\]
	Hence, the extended potentials are optimal for \eqref{Wdual}.
	Since the subsequence choice was arbitrary, this also shows Proposition~\ref{prop:conv}~ii).
\end{proof}

So far we cannot show the convergence of the potentials for $\varepsilon \to 0$ for the fully general case.
Essentially, our approach would require that all $ T_{\mu,\varepsilon}$ are contractive with a uniform constant $\beta<1$, which is not the case.
Note that if we assume that the unregularized potentials satisfying \eqref{add_constraint} are unique, then ii) directly implies convergence of the restricted dual potentials, see also \cite[Thm.~3.3]{Ber20} and \cite{CS94}.
Nevertheless, we always observed convergence in our numerical examples.

\section{Sinkhorn divergence} \label{sec:sink_div}
The regularized OT functional $\OT_\varepsilon$ is biased, i.e., in general $\min_{\nu} \OT_\varepsilon(\nu, \mu) \neq \OT_\varepsilon(\mu, \mu)$.
Hence, the usage as distance measure is meaningless, which motivates the introduction of the \emph{Sinkhorn divergence}
\begin{equation}
S_{\varepsilon}(\mu, \nu) = \OT_{\varepsilon}(\mu, \nu) - \tfrac{1}{2} \OT_{\varepsilon}(\mu, \mu) - \tfrac{1}{2} \OT_{\varepsilon}(\nu, \nu).
\end{equation}
Indeed, it was shown that $S_{\varepsilon}$ is non-negative, bi-convex and metrizes the convergence in law under mild assumptions \cite{FSVATP2018}. 
Clearly, we have $S_0 = \OT$.
By \eqref{dual_kern} and Proposition \ref{lem:PropPot}, we obtain the following corollary.

\begin{corollary}\label{cor:s-d}
	Assume that $K \in C( \X \times \X)$ is symmetric and positive definite. 
	Set 
	$c(x,y) \coloneqq -K(x,y).$
	Then, it holds $S_\infty(\mu,\nu) = \tfrac12 \mathscr{D}_{K}^2(\mu,\nu)$ 
	and the optimal dual potential $\hat \varphi_K$ realizing $\mathscr{D}_K(\mu,\nu)$ is related to
	the uniform limits \smash{$\hat \varphi_\infty, \hat \psi_\infty$ of $\hat \varphi_{\varepsilon},\hat \psi_{\varepsilon}$} in $\OT_{\varepsilon}(\mu, \nu)$ with 
	constraint \eqref{add_constraint} 	by 
	\[
	\hat \varphi_K = \frac{\hat \varphi_\infty
		- \hat \psi_\infty}{\Vert \hat \varphi_\infty - \hat \psi_\infty \Vert_{H_K(\X)}}.
	\]
\end{corollary}

Note that \eqref{gg} already implies that for the chosen $c$ it holds $\hat \varphi_\infty, \hat \psi_\infty \in H_K(\X)$.
By Corollary \ref{cor:s-d}, we have for $c(x,y) \coloneqq -K(x,y)$ that
$S_\infty(\mu,\nu) = \tfrac12 \mathscr{D}_{K}^2(\mu,\nu)$ if $K \in C(\X \times \X)$ is symmetric, positive definite.
For the cost $c(x,y) = \|x-y\|^p$ of the classical $p$-Wasserstein distance, we have already seen in Section~\ref{sec:discrepancies_OT} that $K(x,y) = -c(x,y)$ is not positive definite.
However, at least for $p=1$ the Kernel is conditionally positive definite of order 1 and can be tuned by Proposition~\ref{prop:graef} to a positive definite kernel by adding a constant, which neither changes the value of the discrepancy nor of the optimal dual potential.
More generally, we have the following corollary.

\begin{corollary}\label{cor:conn_disr}
	Let $K \in C(\X \times \X)$ be symmetric, conditionally positive definite of order 1, and let
	$\tilde K$ be the corresponding positive definite kernel in \eqref{cpd_1}.
	Then we have for $c = - \tilde K$ that
	$$
	S_\infty (\mu,\nu) = \tfrac12 \mathscr{D}_{K}^2(\mu,\nu)
	$$
	and for the optimal dual potentials 
	\begin{align*}
	\hat \varphi_\infty(x) 
	&= \int_\X -K(x,y) \dx \nu(y) + \frac12 \int_{\X^2} K \dx (\mu \otimes \nu) 
	+ K(x,\xi)  + \frac12 \bigl( c_\nu - c_\mu - K(\xi,\xi) \bigr),\\
	\hat \psi_\infty(y) 
	&= \int_\X -K(x,y) \dx \mu(x) + \frac12 \int_{\X^2} K  \dx (\mu \otimes \nu)
	+ K(\xi,y) + \frac12 \bigl( c_\mu - c_\nu - K(\xi,\xi)\bigr),
	\end{align*}
	with some fixed $\xi \in \X$ and $c_\mu,c_\nu$ defined as in \eqref{eq:xx}.
\end{corollary}

\begin{proof}
	By Corollary \ref{cor:s-d} and Lemma \ref{lem:cpd_2}, we obtain
	$$
	\mathrm{S}_\infty (\mu,\nu) = \tfrac12 \mathscr{D}_{\tilde K}(\mu,\nu)^2 = \tfrac12 \mathscr{D}_{K}(\mu,\nu)^2.
	$$
	The second claim follows by Proposition~\ref{lem:PropPot}.
\end{proof}

In the following, we want to characterize the convergence of the functional 
$S_\varepsilon(\cdot,\nu)$ in the limiting cases $\varepsilon \to 0$ and $\varepsilon \to \infty$
for fixed $\nu \in {\mathcal P}(\X)$.
Recall that  a sequence $\{F_n\}_{n\in\N}$ 
of functionals $F_n\colon {\mathcal P}(\X) \rightarrow (-\infty,+\infty]$ is said to $\Gamma$-converge to $F \colon  {\mathcal P}(\X) \rightarrow (-\infty,+\infty]$ 
if the following two conditions are fulfilled for every $\mu \in {\mathcal P}(\X)$, see~\cite{Braides02}:
\begin{enumerate}
	\item[i)] $F(\mu) \leq \liminf_{n \rightarrow \infty} F_n(\mu_n)$ whenever  $\mu_n \weakly  \mu$, 
	\item[ii)] there is a sequence $\{\mu_n\}_{n\in\N}$ with $\mu_n \weakly \mu $ and $\limsup_{n \to \infty} F_n(\mu_n) \le F(\mu)$.
\end{enumerate}
The importance of $\Gamma$-convergence relies in the fact that 
every cluster point of minimizers of $\{F_n\}_{n\in\N}$ is a minimizer of $F$.

\begin{proposition}\label{prop:gamma}
	It holds 
	$S_\varepsilon(\cdot,\nu) \xrightarrow{\Gamma}  S_\infty(\cdot,\nu)$ 	as $\varepsilon \to \infty$ 
	and 
	$S_\varepsilon(\cdot,\nu) \xrightarrow{\Gamma}  \OT(\cdot,\nu)$ as $\varepsilon \to 0$.
\end{proposition}

\begin{proof}
	In both cases the $\limsup$-inequality follows from Proposition~\ref{prop:conv} 
	by choosing for some fixed $\mu \in {\mathcal P}(\X)$ the constant sequence $\mu_n = \mu$, $n \in \mathbb N$.
	
	Concerning the $\liminf$-inequality, we first treat the case $\varepsilon \to \infty$.
	Let $\mu_n \weakly \mu$ and $\varepsilon_n \to \infty$.
	Since $\OT_\varepsilon(\mu,\nu)$ is increasing with $\varepsilon$, it holds
	for every fixed $m \in \N$ that
	\begin{align}
	\liminf_{n \rightarrow \infty} S_{\varepsilon_n}(\mu_n, \nu) 
	&=
	\liminf_{n \rightarrow \infty} \left( \OT_{\varepsilon_n}(\mu_n, \nu) - 	\tfrac{1}{2}\OT_{\varepsilon_n} (\mu_n, \mu_n) -\tfrac{1}{2} \OT_{\varepsilon_n} (\nu, \nu) \right)
	\\
	&\geq 
	\liminf_{n \rightarrow \infty} \left( \OT_{m}(\mu_n, \nu) - 	\tfrac{1}{2} \OT_\infty (\mu_n, \mu_n) \right) -\tfrac{1}{2} \OT_\infty(\nu, \nu).
	\end{align}
	Due to the weak continuity of $\OT_m$ and $\OT_\infty$, we obtain
	\[
	\liminf_{n \rightarrow \infty} S_{\varepsilon_n}(\mu_n, \nu) 
	\geq 
	\OT_m(\mu, \nu) - \tfrac{1}{2}\OT_\infty (\mu, \mu) -\tfrac{1}{2} \OT_\infty(\nu, \nu).
	\]
	Letting $m \rightarrow \infty$, Proposition~\ref{prop:conv} implies the $\liminf$-inequality.
	
	Next, we consider $\varepsilon \to 0$.
	Let $\mu_n \weakly \mu$ and $\varepsilon_n \to 0$.
	With similar arguments as above we obtain for any fixed $m \in \mathbb N$ that
	\begin{align}
	\liminf_{n \rightarrow \infty} S_{\varepsilon_n}(\mu_n, \nu) 
	&\geq 
	\liminf_{n \rightarrow \infty} \left( \OT(\mu_n, \nu) - \tfrac{1}{2}\OT_m (\mu_n, \mu_n) \right) - \tfrac{1}{2} \OT_m(\nu, \nu)
	\end{align}
	and weak continuity of $\OT_m$ and $\OT$ implies
	\[
	\liminf_{n \rightarrow \infty} S_{\varepsilon_n}(\mu_n, \nu) 
	\geq 
	\OT(\nu, \mu)  - \tfrac{1}{2}\OT_m (\mu, \mu) 
	-\tfrac{1}{2} \OT_m(\nu, \nu).
	\]
	Using again Proposition~\ref{prop:conv}, we verify the $\liminf$-inequality.
\end{proof}

\section{Numerical approach and examples}\label{sec:numerics_OT}
In this section, we discuss the Sinkhorn algorithm for computing $\mathrm{OT}_{\varepsilon}$ based on the (pre)-dual form~\eqref{pre-dual}
and show some numerical examples.
As pointed out in Remark \ref{rem:restrict}, we can restrict the potentials and the update operator \eqref{eq:condphi} to $\supp(\mu)$ and $\supp(\nu)$, respectively.
In particular, this restriction results in a discrete problem if both input measures are atomic.
For a fixed starting iterate $\psi^{(0)}$, the Sinkhorn algorithm iterates are defined as
\begin{align}
\varphi^{(i+1)} &= T_{\nu,\varepsilon}(\psi^{(i)}),\\
\psi^{(i+1)} &= T_{\mu,\varepsilon}(\varphi^{(i+1)}).
\end{align}
Equivalently, we could rewrite the scheme with just one potential and the following update 
$\psi^{(i+1)} = T_{\mu,\varepsilon}\circ T_{\nu,\varepsilon}(\psi^{(i)})$.
According to Lemma~\ref{prop:PropT}, the operator $T_{\mu,\varepsilon}\circ T_{\nu,\varepsilon}$ is contractive 
and hence the Banach fixed point theorem implies that the algorithm converges linearly.
Note that it suffices to enforce the additional constraint \eqref{add_constraint} after the Sinkhorn scheme by adding an appropriately chosen constant.
Then, the value of $\OT_\varepsilon(\mu, \nu)$ can be computed from the optimal potentials using~\eqref{eq:optimal energy}.
Here, we do not want to go into more detail on implementation issues, since this is not the main scope of this chapter.
The numerical examples merely serve as an illustration of the theoretical results.
All computations in this section are performed using \mbox{GEOMLOSS}, a publicly available PyTorch implementation for regularized optimal transport.
Implementation details can be found in Feydy et al.~\cite{FSVATP2018} and in the corresponding GitHub repository.

\paragraph{Demonstration of convergence results.}
In the following, we present a numerical toy example for illustrating the convergence results from the previous sections.
First, we want to verify the interpolation behavior of $S_\varepsilon(\mu,\nu)$ between $\OT(\mu,\nu)$ and $\mathscr{D}_K(\mu,\nu)$.
We choose $\X = [0,1]$, $c(x,y) = \vert x-y \vert$ and the probability measures $\mu$ and $\nu$ depicted in Fig.~\ref{fig:1}.
The resulting energies $S_\varepsilon(\mu, \nu)$ in log-scale are plotted in the same figure.

\begin{figure}[t]
	\begin{subfigure}[t]{0.325\textwidth}
		\includegraphics[width=\textwidth]{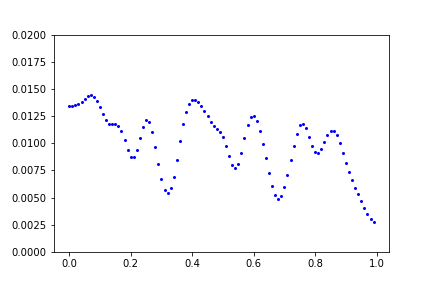}
		\caption{Measure $\mu$}
	\end{subfigure}
	\begin{subfigure}[t]{0.325\textwidth}
		\includegraphics[width=\textwidth]{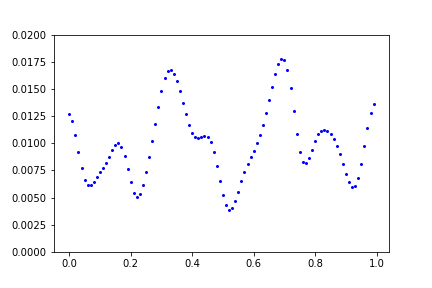}
		\caption{Measure $\nu$}
	\end{subfigure}
	\begin{subfigure}[t]{0.325\textwidth}
		\includegraphics[width=\textwidth]{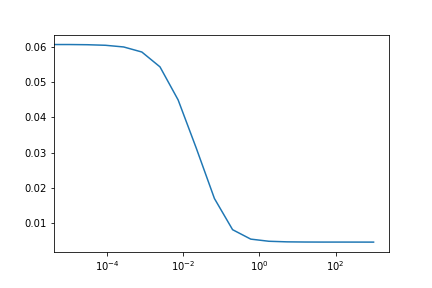}
		\caption{Values $S_\varepsilon(\mu,\nu)$ for increasing~$\varepsilon$}
	\end{subfigure}
	\caption{Energy values between $S_0$ and $S_\infty$ for two given measures on $[0,1]$ and cost function $c(x,y) = |x-y|$.
		Every blue dot corresponds to the position and the weight of a Dirac measure.}\label{fig:1}
\end{figure}

We observe that the values converge as shown in Proposition \ref{prop:conv} and that the change mainly happens in the interval $[10^{-2}, 10^1]$.
Additionally, the numerical results indicate $S_{\varepsilon_1}(\mu, \nu) \leq S_{\varepsilon_2}(\mu, \nu)$ for $\varepsilon_1 > \varepsilon_2$, which is the opposite behavior as for $\OT_\varepsilon$ where the energies increase, see Lemma~\ref{lem:1}~iii).
So far we are not aware of any theoretical result in this direction for $S_{\varepsilon}(\mu, \nu)$.

Next, we investigate the behavior of the corresponding optimal potentials $\hat \varphi_\varepsilon$ and \smash{$\hat \psi_\varepsilon$} in~\eqref{pre-dual}.
The convergence of the potentials as shown in Proposition~\ref{lem:PropPot}~iii) is numerically verified in Fig.~\ref{fig:2.1}.
Further, the corresponding potentials $\hat \varphi_\varepsilon$ are depicted in Fig.~\ref{fig:2.2} and the differences $\hat \varphi_\varepsilon - \hat \psi_\varepsilon$ are depicted in Fig.~\ref{fig:2.3}.
According to Corollary \ref{cor:s-d}, this difference is related to the optimal potential $\hat \varphi_K$ in the dual formulation of the related discrepancy.
The shape of the potentials ranges from something almost linear for small $\varepsilon$ to something more quadratic for large $\varepsilon$.
Again, we observe that the changes mainly happen for $\varepsilon$ in the interval $[10^{-2}, 10^1]$ and that numerical instabilities start to occur for $\varepsilon>10^3$.
For small values of $\varepsilon$, we actually observe numerical convergence and that the relation $\hat \psi_\varepsilon \approx - \hat \varphi_\varepsilon$ holds true, see Fig.~\ref{fig:EqSmall}.
This fits the theoretical findings for $W_1(\mu, \nu)$ in Section~\ref{sec:OT}.

\begin{figure}[t]
	\centering
	\begin{subfigure}[t]{0.325\textwidth}
		\includegraphics[width=\textwidth]{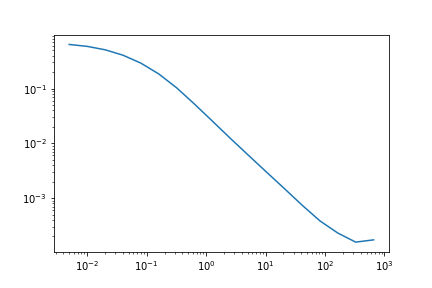}
		\caption{$\sup_{\supp(\mu)} \vert\hat \varphi_{\varepsilon} - \hat \varphi_{\infty} \vert$ for increasing values of $\varepsilon$}
	\end{subfigure}
	\begin{subfigure}[t]{0.325\textwidth}
		\includegraphics[width=\textwidth]{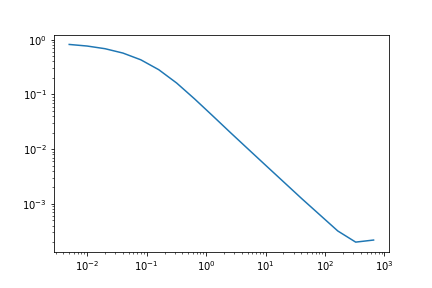}
		\caption{$\sup_{\supp(\nu)} \vert\hat \psi_{\varepsilon} - \hat \psi_{\infty}\vert$ for increasing values of $\varepsilon$}
	\end{subfigure}
	\begin{subfigure}[t]{0.325\textwidth}
		\includegraphics[width=\textwidth]{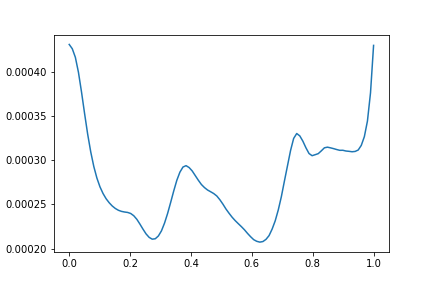}
		\caption{$\hat \varphi_{1e^{-4}} + \hat \psi_{1e^{-4}}$}
		\label{fig:EqSmall}
	\end{subfigure}
	\caption{Numerical verification of Prop.~\ref{lem:PropPot} and of $\hat \psi_\varepsilon \approx - \hat \varphi_\varepsilon$ for small $\varepsilon$.}\label{fig:2.1}
\end{figure}

\begin{figure}[t]
	\centering
	\begin{subfigure}[t]{0.325\textwidth}
		\includegraphics[width=\textwidth]{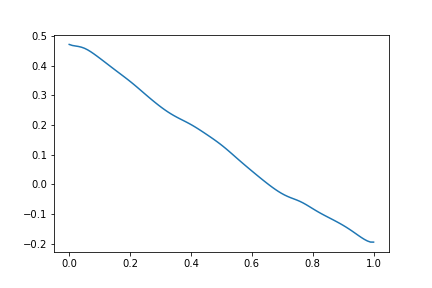}
		\caption{$\hat \varphi_{0.02}$}\label{fig:EqSmall2}
	\end{subfigure}
	\begin{subfigure}[t]{0.325\textwidth}
		\includegraphics[width=\textwidth]{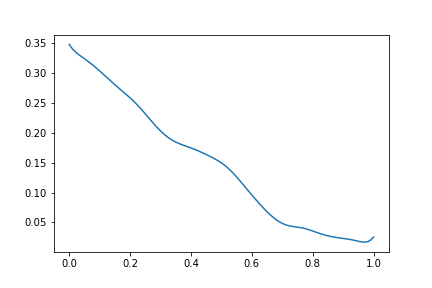}
		\caption{$\hat \varphi_{0.08}$}
	\end{subfigure}
	\begin{subfigure}[t]{0.325\textwidth}
		\includegraphics[width=\textwidth]{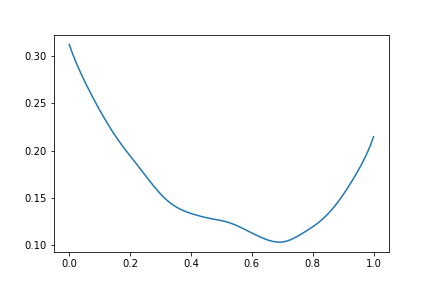}
		\caption{$\hat \varphi_{0.32}$}
	\end{subfigure}
	
	\begin{subfigure}[t]{0.325\textwidth}
		\includegraphics[width=\textwidth]{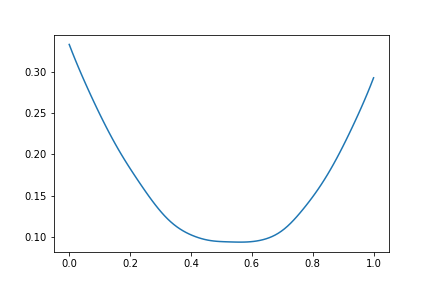}
		\caption{$\hat \varphi_{1.28}$}
	\end{subfigure}
	\begin{subfigure}[t]{0.325\textwidth}
		\includegraphics[width=\textwidth]{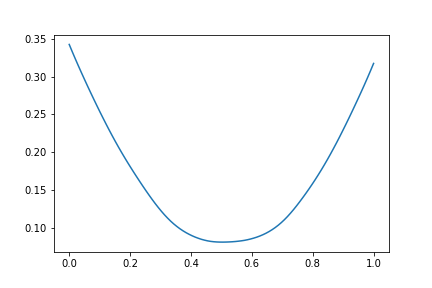}
		\caption{$\hat \varphi_{81.92}$}
	\end{subfigure}
	\begin{subfigure}[t]{0.325\textwidth}
		\includegraphics[width=\textwidth]{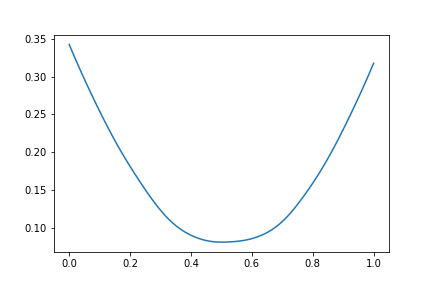}
		\caption{$\hat \varphi_{\infty}$}
	\end{subfigure}
	\caption{Optimal potentials $\hat \varphi_{\varepsilon}$ in $\OT_\varepsilon(\mu,\nu)$ for increasing values of $\varepsilon$.}\label{fig:2.2}
\end{figure}

\begin{figure}[t]
	\begin{subfigure}[t]{0.325\textwidth}
		\includegraphics[width=\textwidth]{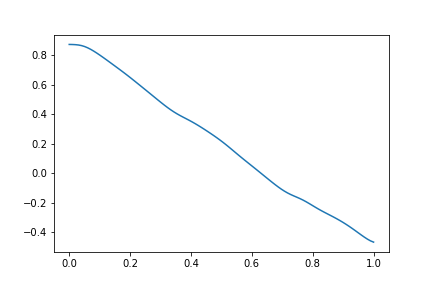}
		\caption{$\hat \varphi_{0.02} - \hat \psi_{0.02}$}\label{fig:EqSmall1}
	\end{subfigure}
	\begin{subfigure}[t]{0.325\textwidth}
		\includegraphics[width=\textwidth]{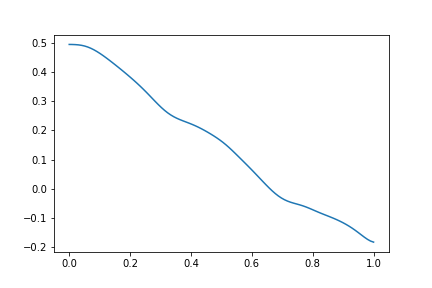}
		\caption{$\hat \varphi_{0.08} - \hat \psi_{0.08}$}
	\end{subfigure}
	\begin{subfigure}[t]{0.325\textwidth}
		\includegraphics[width=\textwidth]{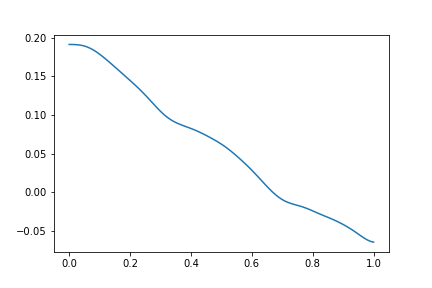}
		\caption{$\hat \varphi_{0.32} - \hat \psi_{0.32}$}
	\end{subfigure}
	
	\begin{subfigure}[t]{0.325\textwidth}
		\includegraphics[width=\textwidth]{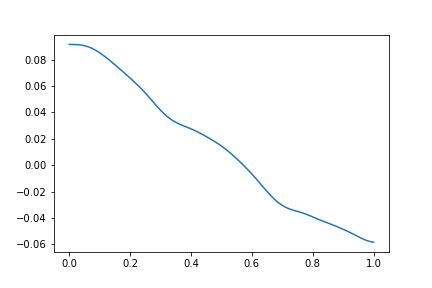}
		\caption{$\hat \varphi_{1.28} - \hat \psi_{1.28}$}
	\end{subfigure}
	\begin{subfigure}[t]{0.325\textwidth}
		\includegraphics[width=\textwidth]{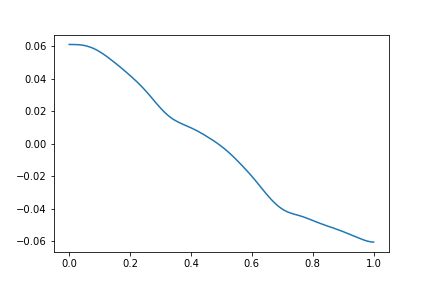}
		\caption{$\hat \varphi_{81.92} - \hat \psi_{81.92}$}
	\end{subfigure}
	\begin{subfigure}[t]{0.325\textwidth}
		\includegraphics[width=\textwidth]{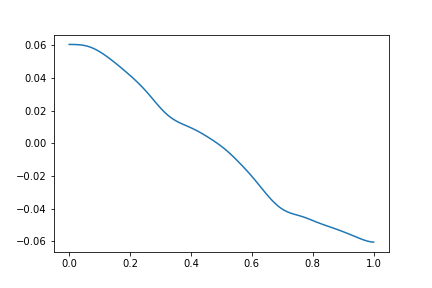}
		\caption{$\hat \varphi_{\infty} - \hat \psi_{\infty}$}
	\end{subfigure}
	\caption{Difference $\hat \varphi_{\varepsilon} - \hat \psi_{\varepsilon}$ of optimal potentials in $\OT_\varepsilon(\mu,\nu)$ for increasing $\varepsilon$,
		where the normalized function $\hat \varphi_{\infty} - \hat \psi_{\infty}$ coincides with the optimal dual potential $\hat \varphi_K$ in the discrepancy
		by Corollary \ref{cor:conn_disr}.}\label{fig:2.3}
\end{figure}
\medskip

\paragraph{Dithering results.}
Now, we want to take a short glimpse at a more involved problem.
In the following, we investigate the influence of using $S_\varepsilon$ with different values $\varepsilon$ as approximation quality measure in dithering.
For this purpose, we choose $\X = [-1,1]^2$, $c(x,y) = \vert x-y \vert$ and $\mu = C \, \exp(-9\Vert x \Vert^2/2) (\lambda \otimes \lambda)$, where $C \in \R$ is a normalizing constant.
In order to deal with a fully discrete problem, $\mu$ is approximated by an atomic measure with $90 \times 90$ spikes on a regular grid.
Then, we approximate $\mu$ with a measure $\nu \in \mathcal{P}_\text{emp}^{400}(\X)$ (empirical measure with 400 spikes) in terms of the following objective function
\begin{equation}\label{eq:Dither}
\min_{\nu \in \mathcal P_\text{emp}^{400}(\X)} S_\varepsilon(\mu,\nu).
\end{equation}
For solving this problem, we can equivalently minimize over the positions of the equally weighted Dirac spikes in $\nu$.
Hence, we need the gradient of $S_\varepsilon$ with respect to these positions.
If $\varepsilon = \infty$, this gradient is given by an analytic expression.
Otherwise, we can apply automatic differentiation tools to the Sinkhorn algorithm in order to compute a numerical gradient, see \cite{FSVATP2018} for more details.
Here, it is important to ensure high enough numerical precision and to perform enough Sinkhorn iterations.
In any case, the gradient serves as input for the L-BFGS-B (Quasi-Newton) method in which the Hessian is approximated in a memory efficient way \cite{BLNZ95}.
The numerical results are depicted in Fig.~\ref{fig:dithering}, where all examples are iterated to high numerical precision.
Numerically, we nicely observe the convergence of $S_\varepsilon(\mu,\hat \nu)$ in the limits $\varepsilon \to 0$ and $\varepsilon \to \infty$ 
as implied from the $\Gamma$-convergence result in Proposition~\ref{prop:gamma}.
Visually, the result using Fourier methods is most appealing. Differences could be caused by the different numerical approaches.
In particular, the minimization of \eqref{eq:Dither} is quite challenging and our applied approach is pretty straight forward without including any special knowledge about the problem.
Noteworthy, the Fourier method uses a truncation of $S_\infty=\tfrac12 \mathscr{D}_K^2$ in the Fourier domain, see \eqref{mercer_2_OT}, namely 
\[
\sum_{k=0}^N \alpha_k  \big| \hat{\mu}_{k}-\hat{\nu}_{k}  \big|^2, \qquad N \coloneqq 128
\]
as target functional, see \cite{Graf:2013fk}.
The value of $S_\infty$ for the Fourier method is slightly larger than the result using optimization of $S_\infty$ directly.
Since the computational cost increases as $\varepsilon$ gets smaller, we suggest to choose $\varepsilon \approx 1$ or to directly stick with discrepancies.
This also avoids that the approximation rates suffer from the so-called curse of dimensionality.

Finally, note that we sampled $\mu$ with a lot more points than we used for the dithering.
If not enough points are used, we would observe clustering of the dithered measure around the positions of $\mu$.
One possibility to avoid such a behavior for $S_\varepsilon$ could be to use the semi-discrete approach described in \cite{GCPB2016}, avoiding any sampling of the measure $\mu$.
In the Fourier based approach, this issue was less pronounced.

\begin{figure}[t]
	\centering
	\begin{subfigure}[t]{0.325\textwidth}
		\includegraphics[width=\textwidth]{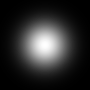}
		\caption{Fixed measure $\mu$.}
	\end{subfigure}
	\begin{subfigure}[t]{0.325\textwidth}
		\includegraphics[width=\textwidth]{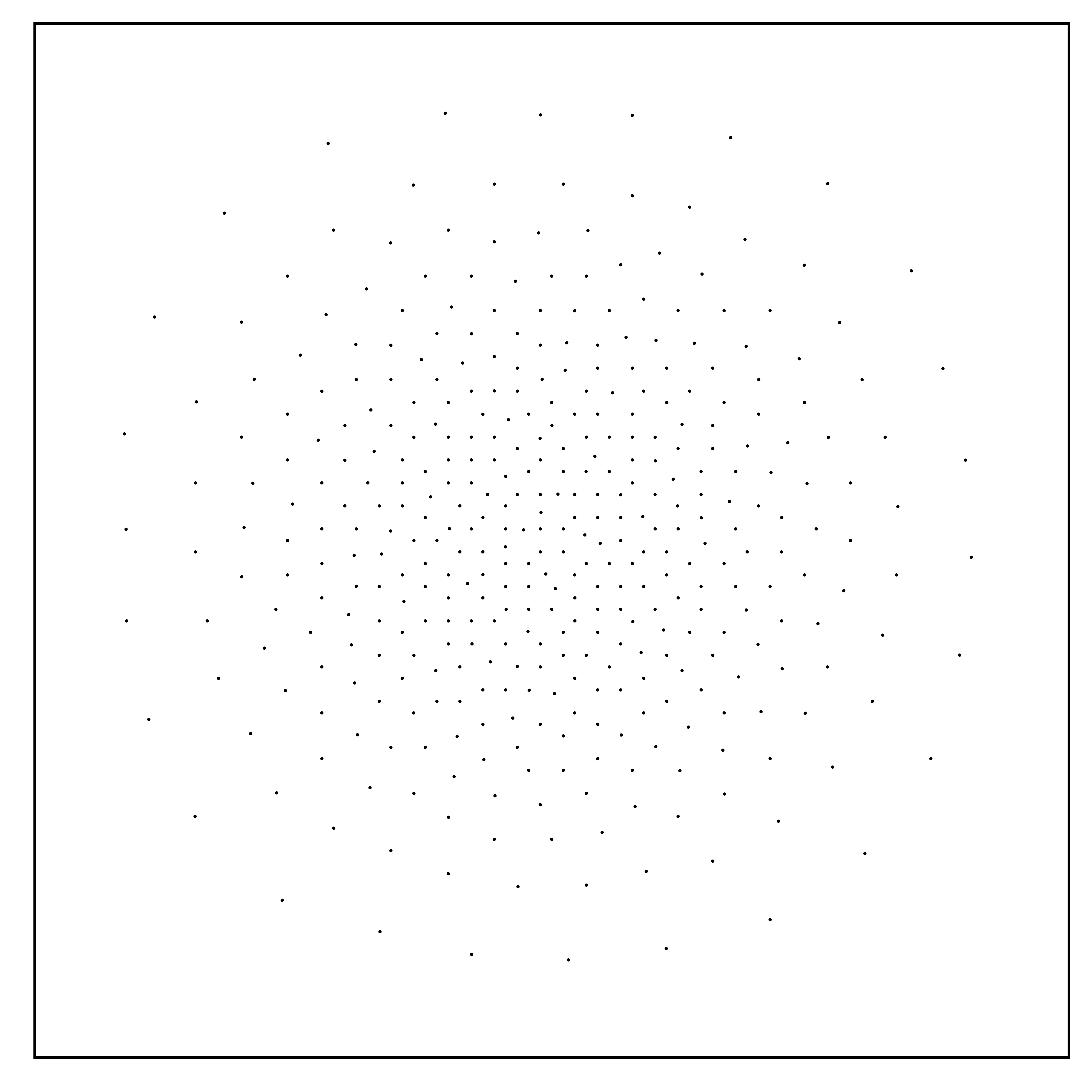}
		\caption{$S_{0.03}(\mu,\hat \nu) = 1.303e^{-3}$.}
	\end{subfigure}
	\begin{subfigure}[t]{0.325\textwidth}
		\includegraphics[width=\textwidth]{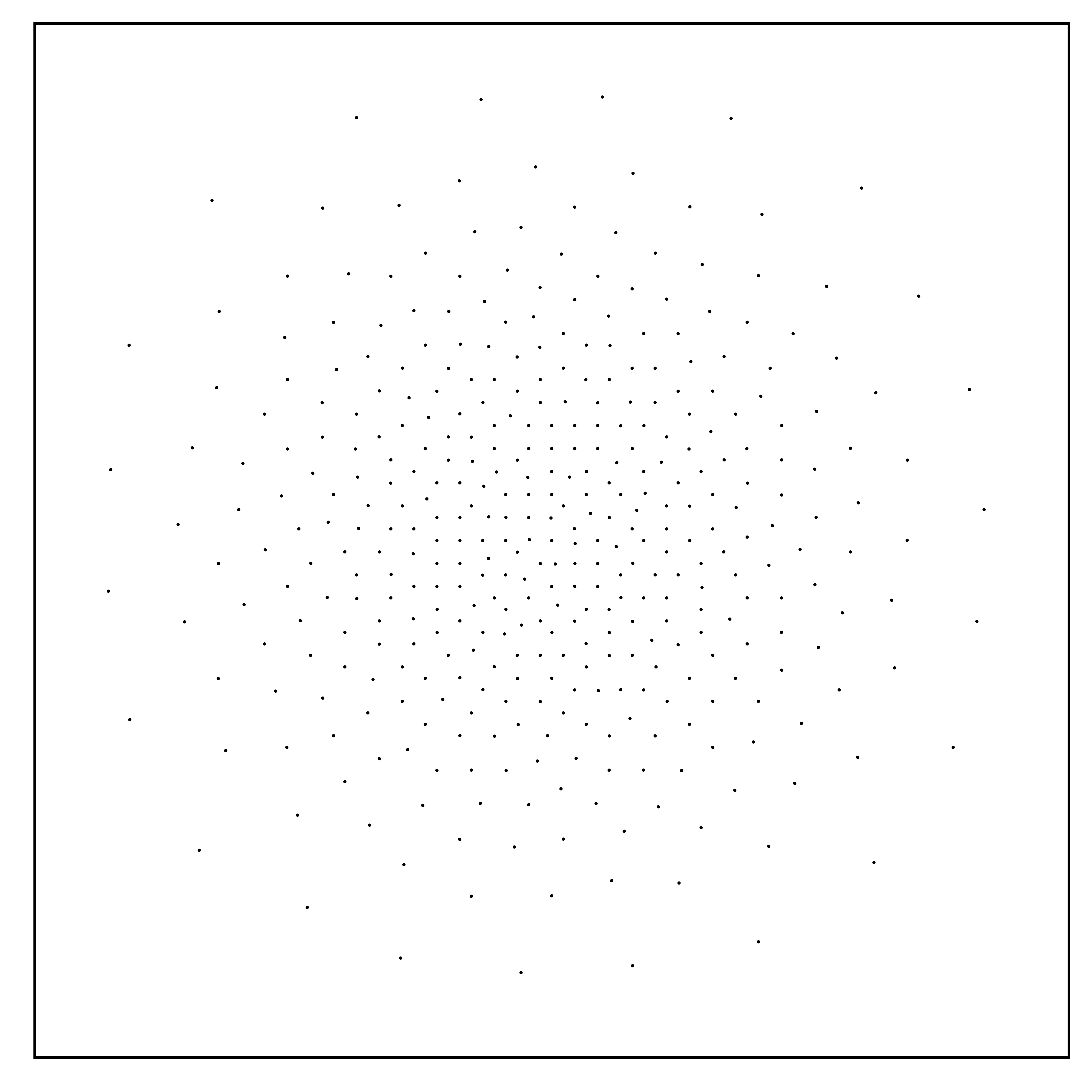}
		\caption{$S_{0.15}(\mu,\hat \nu) = 1.071e^{-4}$.}
	\end{subfigure}
	
	\begin{subfigure}[t]{0.325\textwidth}
		\includegraphics[width=\textwidth]{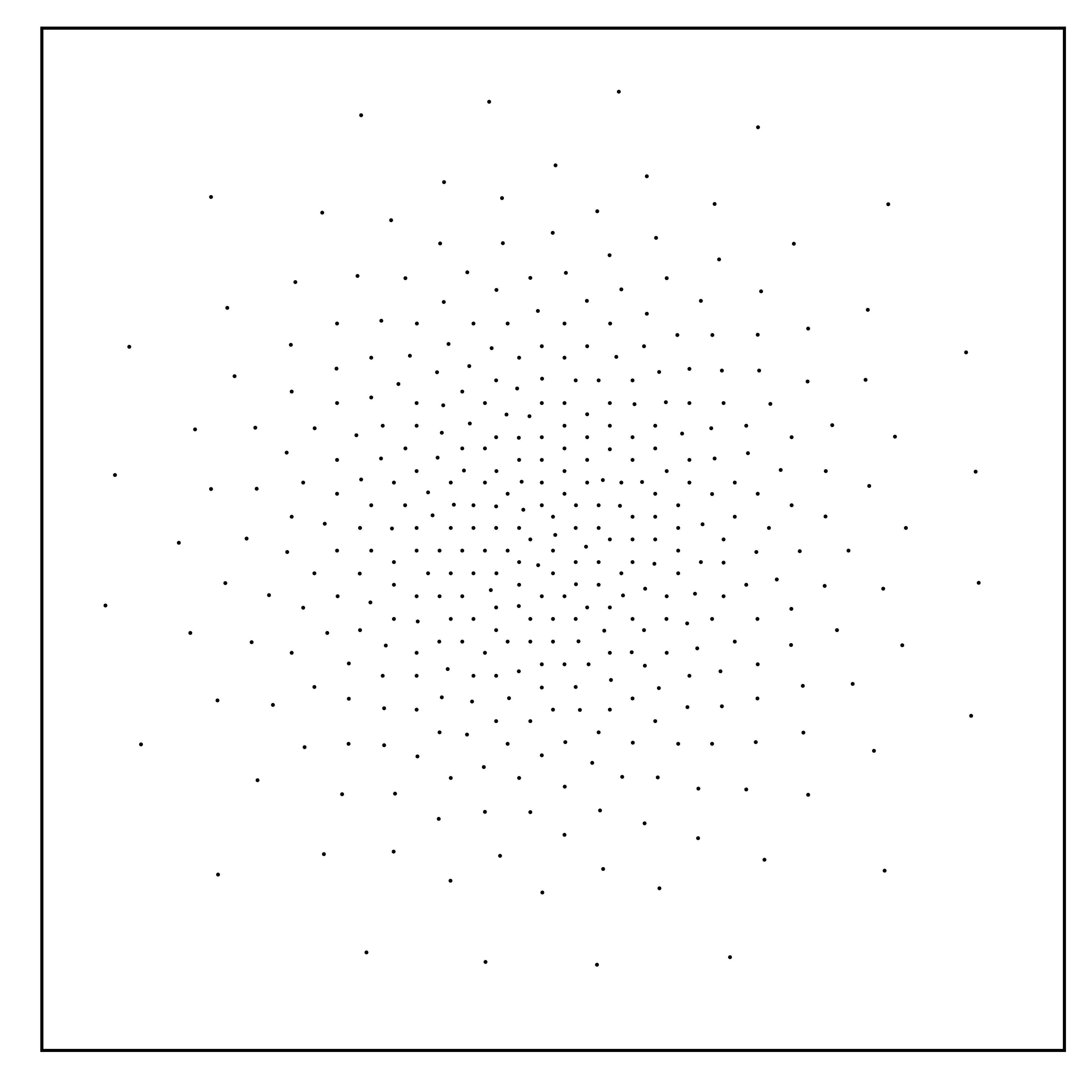}
		\caption{$S_{1.25}(\mu,\hat \nu) = 1.491e^{-5}$.}
	\end{subfigure}
	\begin{subfigure}[t]{0.325\textwidth}
		\includegraphics[width=\textwidth]{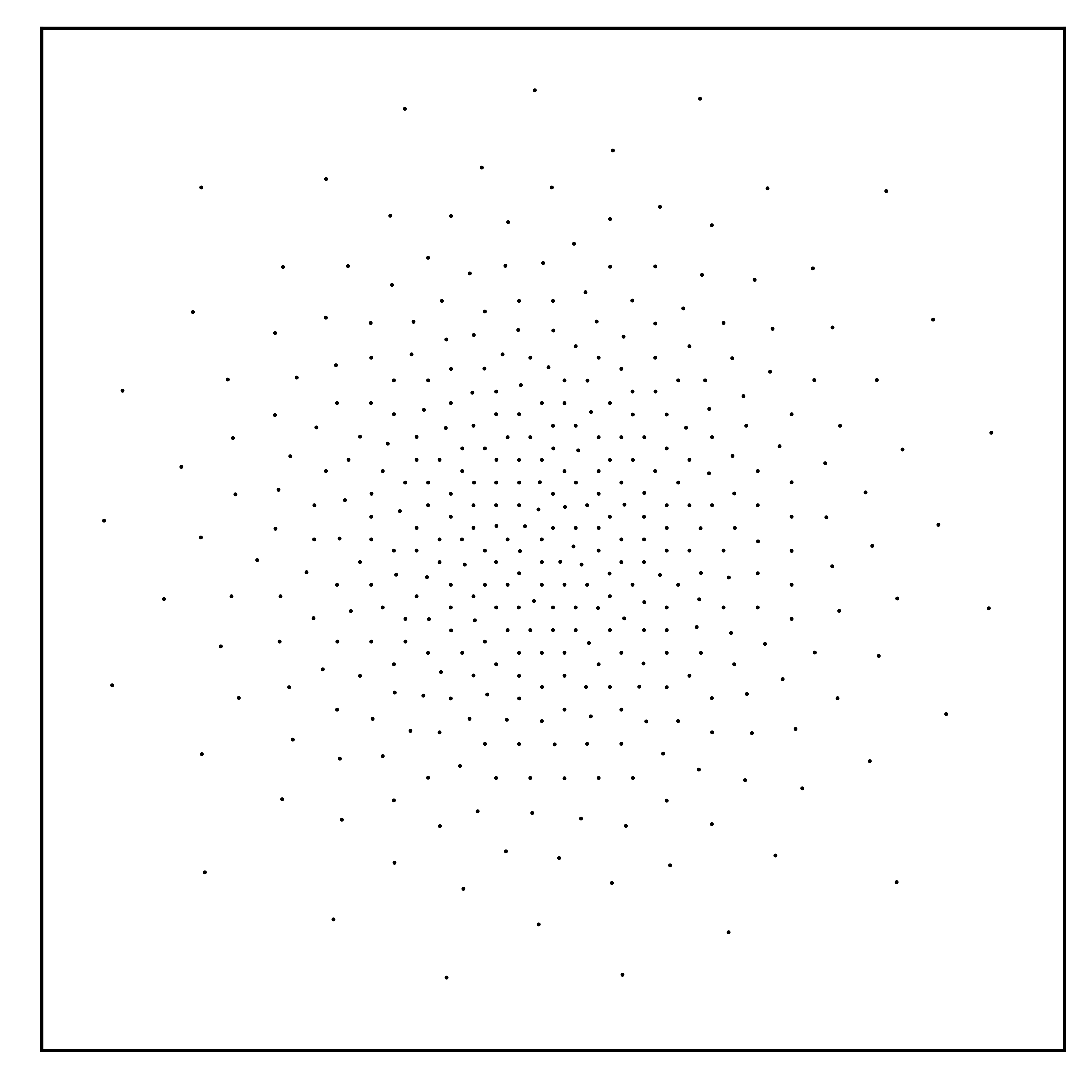}
		\caption{$S_\infty(\mu,\hat \nu) = 1.118e^{-5}$.}
	\end{subfigure}
	\begin{subfigure}[t]{0.325\textwidth}
		\includegraphics[width=\textwidth]{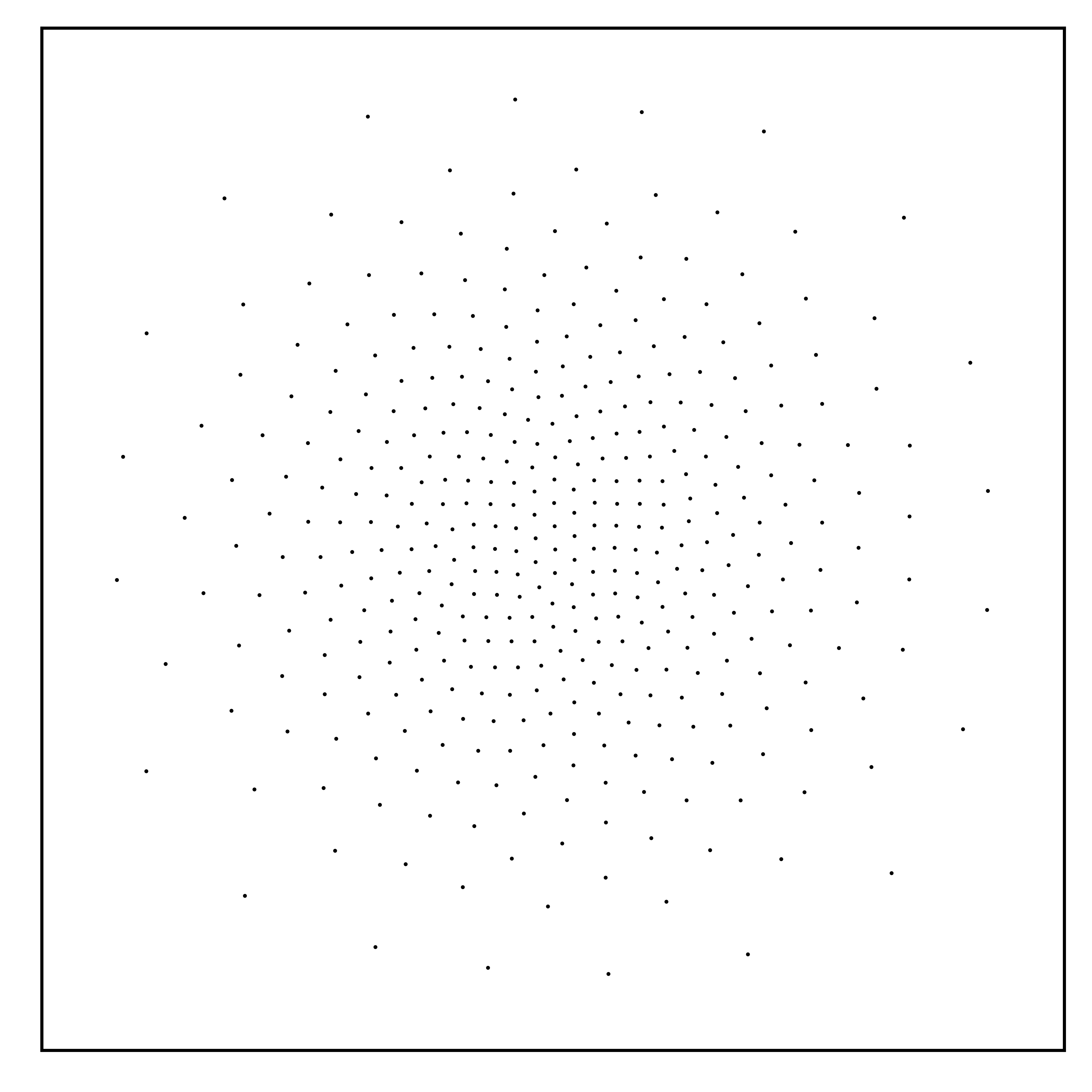}
		\caption{Fourier formulation~\cite{EGNS2019},\newline $S_\infty(\mu,\hat \nu) = 1.156e^{-5}$.}
	\end{subfigure}
	\caption{Optimal approximations $\hat \nu$ and corresponding energies $S_\varepsilon(\mu,\hat \nu)$ for increasing $\varepsilon$.}
	\label{fig:dithering}
\end{figure}

\section{Conclusions}\label{sec:conclusions_OT}
In this chapter, we examined the behavior of the Sinkhorn divergences $S_\varepsilon$ as $\varepsilon \to \infty$ and $\varepsilon \to 0$, with focus on the first case, which leads to discrepancies for appropriate cost functions and kernels.
We considered a quite general scenario of measures involving, e.g., convex combinations of measures with densities and point measures (spikes).
Besides application questions, some open theoretical problem are left.
While $\OT_\varepsilon$ is monotone increasing in $\varepsilon$ for any cost function $c$, we observed numerically for $c(x,y) = \Vert x -y \Vert$ that $S_\varepsilon$ is monotone decreasing.
Further, in Proposition~\ref{lem:PropPot}~ii), we were not able to show convergence of the whole sequence of optimal potentials $\{(\hat \varphi_\varepsilon,\hat \psi_\varepsilon)\}_\varepsilon$ without further assumptions so far. 

\appendix
	\section{Basic theorems}\label{sec:thm}
	
	We  frequently apply the theorem of Arzelà--Ascoli.
	By definition, a sequence $\{f_n\}_{n \in \mathbb N}$ of continuous functions on $\mathbb X$ is 
	\emph{uniformly bounded}, if there exists a constant $M\ge 0$ 
	independent of $n$ and $x$ such that for all $f_n$ and all $x \in \mathbb X$ it holds
	$\left|f_n(x)\right|\leq M$.
	The sequence is said to be \emph{uniformly equi-continuous} if, for every $\varepsilon > 0$, there exists a $\delta > 0$ 
	such that for all functions $f_n$
	$$
	\left|f_{n}(x)-f_{n}(y)\right|<\varepsilon 
	$$
	whenever $d_{\mathbb X}(x,y)< \delta$.
	
	\begin{theorem}(Arzelà--Ascoli)\label{thm_aa}
		Let $\{f_n\}_{n \in \mathbb N}$ be a uniformly bounded and uniformly equi-continuous sequence of continuous functions on $\mathbb X$.
		Then, the sequence has a uniformly convergent subsequence.
	\end{theorem}
	
	For the dual problems, we want to extend continuous functions from $A \subset \X$ to the whole space, which is possible by the following theorem.
	In the standard version, the theorem comes without the bounds, but they can be included directly since $\min$ and $\max$ of two continuous functions are again continuous functions.
	
	\begin{theorem} (Tietze Extension Theorem) \label{thm_tietze}
		Let a closed subset $A \subset \X$ and a continuous function $f\colon A \to \R$ be given.
		If $g,h \in C(\X)$ are such that $g \leq h$ and $g(x) \leq f(x) \leq h(x)$ for all $x \in A$, then there exists a continuous function $F \colon \X \to \R$ 
		such that $F(x) = f(x)$ for all $x \in A$ and $g(x) \leq F(x) \leq h(x)$ for all $x \in \X$.
	\end{theorem}

\bibliographystyle{abbrv}
\bibliography{./references_clean}
\end{document}